\documentclass[12pt,a4paper]{amsart}
\usepackage{fullpage}
\usepackage{mathrsfs}
\usepackage{CJK}
\usepackage{geometry}   
\usepackage{setspace}
\usepackage{amsfonts}
\usepackage{xy}
\usepackage{cite}
\usepackage{amssymb}
\usepackage{color, latexsym}
\usepackage{graphicx}
\usepackage{amsbsy,amscd}
\usepackage{fancyhdr}
\usepackage[colorlinks=true,
           citecolor=blue,
           linkcolor =red]{hyperref}
\xyoption{all}
\allowdisplaybreaks

\fontsize{9pt}{9pt}\selectfont

\def\bmu{\boldsymbol{\mu}}
\def\bgl{\boldsymbol{\lambda}}

\def\bi{\mathbf{i}}
\def\bk{\mathbf{k}}
\def\bl{\boldsymbol{\ell}}
\def\bq{\mathbf{Q}}

\def\ci{\mathcal{I}}

\def\gl{\lambda}

\def\bz{\mathbb{Z}}

\def\fgl{\mathfrak{gl}}
\def\fg{\mathfrak{g}}
\def\fs{\mathfrak{s}}
\def\fS{\mathfrak{S}}
\def\ft{\mathfrak{t}}
\def\h{\mathcal{H}}
\def\Hom{\mathrm{Hom}}
\def\res{\mathrm{res}}
\def\std{\mathrm{Std}}

\def\mpn{\mathscr{P}_{m,n}}
\def\la{\langle}
\def\ra{\rangle}

\newdimen\hoogte    \hoogte=14pt    
\newdimen\breedte   \breedte=14pt   
\newdimen\dikte     \dikte=0.5pt    
\newenvironment{point}[2]%
  {\vspace{0.5\jot}\ifx*#2\let\pointlabel\relax\else\def\pointlabel{#2}\fi
   \refstepcounter{equation}\trivlist
   \item[\hskip\labelsep\theequation.
         \ifx\pointlabel\relax\else\space\pointlabel\space\fi]
   \ignorespaces #1
  \vspace{0.5\jot}}{\relax}

\numberwithin{equation}{section}

\swapnumbers
\newtheorem{theorem}[equation]{Theorem}
\newtheorem{lemma}[equation]{Lemma}
\newtheorem{proposition}[equation]{Proposition}

\theoremstyle{definition}
\newtheorem{definition}[equation]{Definition}

\theoremstyle{remark}
\newtheorem{remark}[equation]{Remark}

\begin{document}
\begin{CJK*}{GBK}{song}
\setlength{\itemsep}{1\jot}
\fontsize{13}{\baselineskip}\selectfont
\setlength{\parskip}{0.3\baselineskip}
\vspace*{0mm}
\title[Schur-Sergeev duality]{\fontsize{9}{\baselineskip}\selectfont Schur-Sergeev duality for Ariki-Koike algebras}
\author{Deke Zhao}
\date{}
\address{\bigskip\hfil\begin{tabular}{l@{}}
          School of Applied Mathematics\\
            Beijing Normal University at Zhuhai, Zhuhai, 519087\\
             China\\
             E-mail: \it deke@bnuz.edu.cn \hfill
          \end{tabular}}
\thanks{Supported by the National Natural Science Foundation of China (Grant No. 11871107)}
\vspace*{-3mm}
\begin{abstract}Let $U_q(\fg)$ be the quantized superalgebra  of $\fg=\fgl(k_1|\ell_1)\oplus\cdots\oplus\fgl(k_m|\ell_m)$ and let $\h$ be the Ariki-Koike algebra. We define a right $\h$-action on the $n$-fold tensor (super) space of the vector representation of  $U_q(\fg)$ and prove the Schur-Sergeev duality between $U_q(\fg)$ and $\h$.
\end{abstract}
\maketitle
\vspace*{-5mm}
\section{Introduction}
Let $k$ be a positive integer. It is known that the group $\mathrm{GL}(k,\mathbb{C})^{\times n}$ acts diagonally on the $n$-fold tensor space of $(\mathbb{C}^{k})^{\otimes n}$ of $\mathbb{C}^k$:  \begin{equation*}  g(v_1\otimes \cdots\otimes v_n)=g(v_1)\otimes \cdots\otimes g(v_n) \end{equation*} for $g\in \mathrm{GL}(k,\mathbb{C})$ and $v_i\in \mathbb{C}^{k}$.  There is also a natural action of the symmetric group $\fS_n$ on $(\mathbb{C}^{k})^{\otimes n}$, given by permuting the factors: \begin{equation*} s_i(v_1 \otimes \cdots \otimes v_{i} \otimes\negmedspace v_{i+1} \otimes  \cdots \otimes v_n) =  v_1 \otimes
\cdots \otimes  v_{i+1} \otimes  v_{i} \otimes  \cdots \otimes  v_n,1\leq i\leq n-1\end{equation*} for the simple transpositions $s_i=(i,i+1)\in\fS_n$ and $v_i\in\mathbb{C}^{k}$. Schur \cite{Schur-d,Schur} showed that $\mathrm{GL}(k,\mathbb{C})$ and $\fS_n$ are mutual centralizers of each other in $\mathfrak{gl}(n)=\mathrm{End}_{\mathbb{C}}(V^{\otimes n})$, which now is known as the classical Schur-Weyl reciprocity, and obtained the Frobenius formula \cite{Frobenius} by applying this reciprocity.

After Schur's classical work, Schur--Weyl reciprocity has been extended to various settings. Here we only review briefly the followings inspiring the present work:
\begin{enumerate}\setlength{\itemsep}{1\jot}
\item Let $U_q(\mathfrak{gl}(k))$ be the quantized enveloping algebra of $\mathfrak{gl}(k)$ and $\mathcal{H}_n(q^2)$ the Iwahori-Hecke algebra of type $A$. In \cite{Jimbo}, Jimbo defined an $\mathcal{H}_n(q^2)$-action on the $n$-fold tensor space of the natural representation of $U_q(\mathfrak{gl}(k))$ and showed the quantum Schur-Weyl reciprocity between $U_q(\mathfrak{gl}(k))$ and $\mathcal{H}_n(q^2)$.

\item Let $\mathbb{C}^{k|\ell}$ be the superspace with dimension $k|\ell$  and $\mathfrak{gl}(k|\ell)$ the general linear Lie superalgebra, that is,  $\mathfrak{gl}(k|\ell)=\mathrm{End}_{\mathbb{C}}(\mathbb{C}^{k|\ell})$. Then $(\mathbb{C}^{k|\ell})^{\otimes n}$ is a $\mathfrak{gl}(k|\ell)$-module by letting\begin{equation*}g(\negmedspace v_1\negmedspace\otimes\negmedspace\cdots\negmedspace\otimes\negmedspace v_n\negmedspace)\negmedspace=\negmedspace g(v_1)\negmedspace\otimes\negmedspace \cdots\negmedspace\otimes\negmedspace v_n\negmedspace+\negmedspace\sum_{i=2}^{n}(-\!1)^{\overline{g}\overline{v_1\otimes\cdots\otimes v_{i-1}}}v_1\negmedspace\otimes\negmedspace \cdots\negmedspace \otimes\negmedspace g(v_{i})\negmedspace\otimes\negmedspace \cdots\negmedspace\otimes\negmedspace v_{n},\end{equation*} where $g\in \mathfrak{gl}(k|\ell)$ and $v_i\in \mathbb{C}^{k|\ell}$  ($1\leq i\leq n$) are homogeneous  with degree $\bar{g}$ and $\bar{v_i}$ respectively. There is also an $\fS_n$-action on $(\mathbb{C}^{k|\ell})^{\otimes n}$ given by
\begin{equation*}\label{Equ:SBR-sign} s_i(v_1\negmedspace\otimes\negmedspace\cdots\negmedspace\otimes\negmedspace v_{i}\negmedspace\otimes\negmedspace v_{i+1}\negmedspace\otimes\negmedspace \cdots\negmedspace\otimes\negmedspace v_n)\negmedspace=\negmedspace(-\!1)^{\overline{v}_i\overline{v}_{i+1}}v_1\negmedspace\otimes\negmedspace
\cdots\negmedspace\otimes\negmedspace v_{i+1}\negmedspace\otimes\negmedspace v_{i}\negmedspace\otimes\negmedspace \cdots\negmedspace\otimes\negmedspace v_n,1\leq i<n,\end{equation*}
 where $v_i$ ($1\leq i\leq n$) are homogeneous of $\mathbb{C}^{k|\ell}$. Then the Schur-Weyl duality between $\mathfrak{gl}(k|\ell)$ and $\mathbb{C}\mathfrak{S}_n$ was established first by Sergeev in \cite{Serg} and then in more detail by Berele and Regev \cite{B-Regev}. This dualiy is sometimes
called the Schur--Sergeev duality in the literature (see e.g. \cite{Chen-Wang}).

\item\label{Item:AST} Let $\h$ be the Ariki--Koike algebras, i.e., the cyclotomic Hecke algebras of type $G(m,1,n)$ and $U_q(\overline{\fg})$ the quantized enveloping algebra of a Levi subalgebra $\overline{\fg}=\fgl(k_1)\oplus\cdots\oplus\fgl(k_m)$ of $\fgl(k)$ with $k=\sum_{i=1}^mk_i$. Based on Jimbo's work \cite{Jimbo}, Ariki et al \cite{ATY} gave a Schur--Weyl reciprocity between $U_q(\overline{\fg})$ and $\h$ for all $k_i=1$; Sakamoto and Shoji \cite{SS} and Hu \cite{Hu} established independently the reciprocity for the general case by applying completely different constructions of the $\h$-action on tensor space of the natural representation of $U_q(\overline{\fg})$ and by applying different arguments.

\item  Let $U_q(\mathfrak{gl}(k|\ell))$ be the quantized enveloping superalgebra of $\mathfrak{gl}(k|\ell)$. The Schur-Weyl duality between $U_q(\mathfrak{gl}(k|\ell))$ and $\mathcal{H}_n(q^2)$ was shown independently by Moon \cite{Moon} and by Mitsuhashi \cite{Mit} via different approaches, which is a quantum analogue of the Schur-Sergeev duality.
\end{enumerate}

Now let $\ell_i (i=1, \ldots, m)$ be non-negative integers with $\sum_{i=1}^m\ell_i=\ell$.  Let $\fg=\fgl(k_1|\ell_1)\oplus\cdots\oplus\fgl(k_m|\ell_m)$ and denote by $U_q(\fg)$ its quantized enveloping superalgebra. Motivated by the aforementioned works, the purpose of this paper is to present a Schur-Weyl reciprocity between the quantum superalgebra $U_q(\fg)$ and the Ariki-Koike algebra $\mathcal{H}$ by adapting Sakamoto and Shoji's argument in \cite{SS}, which is called the Schur-Sergeev duality and unifies these works. More precisely, let $(\Psi^{\otimes n}, V^{\otimes n})$ be the vector representation of the quantized enveloping superalgebra $U_q(\fgl(k|\ell))$  of $\fgl(k|\ell)$ over $\mathbb{K}=\mathbb{C}(q,\mathbf{Q})$ (see \S\ref{Sub-Sec:Vector-reps}). Note that $\fg$ can be viewed as a subalgebra of Lie superalgebra $\fgl(k|\ell)$, which enable us to yield a $U_q(\fg)$-action on $V^{\otimes n}$ via the restriction of $\Psi^{\otimes n}$, which is also denoted by $(\Psi^{\otimes n}, V^{\otimes n})$. By extending Moon and  Mitsuhashi's \textit{loc.\,cit.}~works, we define an $\h$-action on $V^{\otimes n}$, which is proved to be an $\h$-representation $(\Phi, V^{\otimes n}$) (Theorem~\ref{Them:Phi-reps}).  It is not hard to show that $\Phi$ actually commutes with   $\Psi^{\otimes n}$, while we have to make much effort to show that $\Phi(\h)$ and $\Psi^{\otimes n}(U_q(\fg))$ are mutually the full centralizer algebras of each other by applying the representations of Ariki-Koike algebras. Therefore we can prove the Schur-Sergeev duality for $\h$  (Theorem~\ref{Them:Schur-Weyl}).

 Let us remark that Hu \cite{Hu} showed the Schur-Weyl reciprocity for the Ariki-Koike algebras by different approach. It would be interesting to give an alternate proof of the Schur-Sergeev  duality by adapting Hu's argument.

We discuss below several questions motivated by the present work.

The classical Schur algebras appeared in an implicit form in Schur's
remarkable article \cite{Schur}. Schur's ideas were presented by J.A. Green in a modern way in \cite{Green}, where their significance for
representation theory of general linear and symmetric groups over any infinite
field was shown. Most of the further generalizations follow the ideas of this
engrossing book. Note that the classical Schur algebras may be viewed as algebras of endomorphisms of tensor space commuting with the action of $\fS_n$ and can be defined over the integers. Dipper and James introduced the $q$-Schur algebras type $A$ as algebras of endomorphisms of tensor space commuting with the action of $H_n(q)$ in \cite{DJ89} (see \cite{Luo-Wang} for uniform formulation of $q$-Schur algebras of arbitrary finite type). Using the cellularity of Ariki-koike algebras, Dipper, James and Mathas introduced the cyclotomic $q$-Schur algebras related to $\h$ along Dipper and James's work \cite{DJ89}. In the super setting,  the Schur superalgebras were introduced in  Muir's  PhD thesis \cite{Muir}, the Schur $q$-superalgebras were introduced Du and Rui in \cite{DR} and their representations were studied extensively by  Du and his coauthors (see e.g. \cite{DGW14,DGW17,DLZ}). Therefore, it is very interesting to give a super analogue of cyclotomic $q$-Schur algebras and study their structure and representations extensively.

In \cite{Z-C}, we will introduce the cyclotomic $q$-Schur superalgebras and show that they enjoy many properties in common with the cyclotomic $q$-Schur algebras and quantum Schur superalgebras.  Let us remark that Deng et al. \cite{Deng-Du-Yang} recently introduce the slim cyclotomic $q$-Schur algebras, which is a new version of cyclotomic $q$-Schur algebras. It would be very interesting to formulate a super-version of the slim cyclotomic $q$-Schur algebras.

Based on the quantum Schur-Weyl reciprocity, Ram \cite{Ram} gave a $q$-analogue of Frobenius formula for the characters of the Iwahori-Hecke algebras of type $A$. A super Frobenius formula for the characters of the Iwahori-Hecke algebras of type $A$ was given by Mitsuhashi in \cite{M2010} by applying  the super quantum Schur-Weyl reciprocity. An extension of Frobenius formula for the characters of cyclotomic Hecke algebra of type $G(m,1,n)$ is found in \cite{S} by applying the Schur-Weyl reciprocity between cyclotomic Hecke algebras and quantum algebras given in \cite{SS}.  Based on \cite{S,M2010} and the present work, we will give a super Frobenius formula for the characters of the characters of cyclotomic Hecke algebras in  \cite{Z-F}.

In 2013, Regev \cite{Regev-2013} presented a surprising beautiful formula for the characters of the symmetric group super representations by applying the Schur--Sergeev duality and the combinatorial theory of Lie superalgebras, which is developed in \cite{B-Regev}. Based on Moon's work \cite{Moon} and Mitsuhashi's work \cite{Mit}, the author gives a quantum analogue of Regev formula and derives a simple formula for the Hecke algebra super character on the exterior algebra in \cite{Zhao}. Motivated by these works, a natural problem is to provide a Regev formula for Ariki-Koike algebras.

Combining the  Schur-Sergeev duality established Sergeev and Berele--Regev and ideas of Serganova \cite{Serganova-88}, Brundan and Kujiwa \cite{BK} obtained a new proof of the Mullineux conjecture, which was first conjectured by Mullineux in \cite{Mullineux} and proved by Ford and Kleshchev in \cite{Ford-Kleshchev}.  Very recently, based on their study on the polynomial representations of the quantum (super) hyperalgebra associated with the quantum enveloping superalgebra  of $\mathfrak{gl}(k|\ell)$,  Du et al.   \cite{DLZ} present a new proof of the quantum version of the Mullineux conjecture for Hecke algebra of type $A$, which was first proved by Brundan \cite{B-98} along Kleshchev's classical works. Thus it would be very interesting to  reinterpret the Mullineux involution for Ariki-Koike algebra \cite{Jacon-L} via representation theory of cyclotomic $q$-Schur superalgebras, which is our last motivation of this paper. Furthermore, one might expect that this interpretation would helpful to understand Dudas and Jacon's work \cite{Dudas-Jacon} and to enhance our understanding on  wall-crossing functors for representations of rational Cherednik algebras introduced by Losev in \cite{Losev}.

 This paper is organized as follows. We begin in Section~\ref{Sec:Quantum-superalgebras} with the definition of quantized enveloping superalgebra and its vector representations, and  fix some combinatoric notations. Section~\ref{Sec:Sign-q-reps} devotes to introduce the sign $q$-permutation representation of Ariki-Koike algebras on tensor product of superspace. Finally,  we establish the Schur-Sergeev duality between the quantum superalgebra $U(\fg)$ and the Ariki-Koike algebras in last section.

 Throughout the paper, we assume that  $\mathbb{K}=\mathbb{C}(q,\bq)$ the field of rational function in indeterminates $q$ and  $\bq=(Q_1, \ldots, Q_m)$.
 For fixed non-negative $k,\ell$ with $k+\ell>0$, we define the parity function  $i\mapsto \overline{i}$ by \begin{equation*} \overline{i}=\left\{\begin{array}{ll}\overline{0}, &\hbox{ if }1\le i\leq k;\\\overline{1}, & \hbox{ if }k<i\leq k+\ell.\end{array}\right.\end{equation*}
Assume that $k_1, \ldots, k_m$, $\ell_1, \ldots, \ell_m$ are non-negative integers satisfying  $\sum_{i=1}^mk_i=k$, $\sum_{i=1}^m\ell_i=\ell$ and denote by $\bk=(k_1, \ldots, k_m)$, $\bl=(\ell_1, \ldots, \ell_m)$. For $i=1, \ldots, m$, we define $d_i=\sum_{j\leq i}k_j+\ell_j$.

\noindent\textbf{Acknowledgements.} Part of this work was carried out while the author was visiting Northeastern University at Qinhuangdao and the Chern Institute of Mathematics (CIM) in Nankai University and he would like to thank Professors Chengming Bai, Ming Ding and Yanbo Li for their hospitalities during his visits. A first manuscript of this paper was announced at the "Academic Seminar on Algebra and Cryptography'' at Hubei University (Wuhan, May 2018), the author would like to thank Professors Xiangyong Zeng, Yunge Xu and  Yuan Chen for their hospitality. 

\section{Preliminaries}\label{Sec:Quantum-superalgebras}
In this section we begin with the definition of the quantum superalgebra $U_q(\mathfrak{gl}(k|\ell))$, i.e., the quantized universal enveloping algebra of the general linear Lie superalgebra $\mathfrak{gl}(k|\ell)$, and define its vector representation. Then we introduce the Lie superalgebra $\fg=\mathfrak{gl}(k_1,\ell_1)\oplus\cdots\oplus\mathfrak{gl}(k_m,\ell_m)$ and its quantized universal enveloping algebra. Finally, we  fix some notations of combinatorics. Note that the Serre-type presentations of the quantization of $\mathfrak{gl}({k|\ell})$ were obtained by various authors all roughly at about the same time (see e.g. \cite{KT,FLV,S92,S93,Zhang}). In this paper we adopt a definition appeared in \cite{Zhang} to quote results there.

\begin{point}{}*
By a superspace  we means a $\mathbb{Z}_2$-graded vector space $U$ over $\mathbb{C}$, namely a $\mathbb{C}$-vector space with a decomposition
into two subspaces $U = U_{\bar{0}}\oplus U_{\bar{1}}$.  A nonzero element $u$ of
$U_i$ will be called \textit{homogeneous} and we denote its degree by
$\overline{u}={i}\in \mathbb{Z}_2$. We will view  $\mathbb{C}$ as a superspace concentrated in degree 0.

Given   superspaces $U$ and $W$, we view the direct sum $U\oplus W$
and the tensor product $U\otimes_{\mathbb{C}} W$  as superspaces with $(U\oplus
W)_i = U_i\oplus W_i$, and $(U\otimes_{\mathbb{C}} W)_i= U_{\bar{0}}\otimes_{\mathbb{C}} V_i\oplus
U_{\bar{1}}\otimes_{\mathbb{C}} W_{\bar{1}-i}$  for $i\in\mathbb{Z}_2$. With this grading,
$U\otimes_\mathbb{C}W$ is called the \textit{tensor space} of $U$ and
$W$ and is denoted by $U\otimes W$. Also, we make the
vector space $\Hom_\mathbb{C}(U, W)$ of all $\mathbb{C}$-linear maps from $U$ to $W$
into a superspace by setting that $\Hom_{\mathbb{C}}(U, W)_i$ consists of all
the $\mathbb{C}$-linear maps $f: U \rightarrow W$ with $f(U_j)\subseteq
W_{i+j}$ for $i, j\in \mathbb{Z}_2$. Elements of
$\Hom_{\mathbb{C}}(U, W)_{\bar{0}}$ (resp.~$\Hom_{\mathbb{C}}(U, W)_{\bar{1}}$) will be referred to as
\textit{even } (resp. \textit{odd}) \textit{linear maps}.

Recall that a \textit{superalgebra} $A$ is both a superspace
and  an associative  algebra  with identity such that
$A_iA_j\subseteq A_{i+j}$ for $i, j \in \mathbb{Z}_2$. Given  two  superalgebras $A$ and $B$, the tensor space $A\otimes B$  is again a superalgebra with
the inducing grading and multiplication given by
\begin{align*}(a_1 {\otimes} b_1)(a_2 {\otimes} b_2) =
(-1)^{\overline{b}_1\overline{a}_2}a_1a_2{\otimes} b_1b_2, \text{ for } a_i\in A \text{ and } b_i\in B.\end{align*}
Note these and other such expressions only make sense for
homogeneous elements. Observe that the $n$-fold tensor space $A^{\otimes n}:=A \otimes A \otimes \cdots \otimes A$ of $A$ is well-defined for all $n$.
Furthermore, if $\phi\in \mathrm{End}_{\mathbb{C}}(A)$ and $\mathrm{End}_{\mathbb{C}}(B)$ are homogeneous endomorphisms then the tensor $\phi\otimes \psi$ is defined as follows:
\begin{equation}\label{Equ:Tensor-hom}
  (\phi\otimes \psi)(a\otimes b):=(-1)^{\bar{a}\bar{\psi}}\phi(a)\otimes \psi(b)
\end{equation}
\end{point}
\begin{point}{}*
The Lie superalgebra $\mathfrak{gl}(k|\ell)$ is the $(k+\ell)\times (k+\ell)$ matrices with $\mathbb{Z}_2$-gradings given by \begin{eqnarray*} \mathfrak{gl}(k|\ell)_{\bar 0}&=&\left\{\left.\left(\begin{array}{cc}\mathbf{A} & \mathbf{0} \\ \mathbf{0} & \mathbf{D}\end{array}\right)\right|\mathbf{A}=(a_{ij})_{1\leq i,j\leq k}, \mathbf{D}=(d_{ij})_{k< i,j\leq k+\ell}\right\},\\ \mathfrak{gl}(k|\ell)_{\bar 1}&=&\left\{\left.\left(\begin{array}{cc}\mathbf{0}& \mathbf{B} \\ \mathbf{C} & \mathbf{0}\end{array} \right)\right|\mathbf{B}=(b_{ij})_{1\leq i\leq k}^{k<j\leq k+\ell}, \mathbf{C}=(c_{ij})_{k<i\leq k+\ell}^{1\leq j\leq k}\right\}                                      \end{eqnarray*}
and Lie bracket product defined by $$[\mathbf{X},\mathbf{Y}]:=\mathbf{XY}-(-1)^{\overline{\mathbf{X}}\,\overline{\mathbf{Y}}}\mathbf{YX}$$ for homogeneous $\mathbf{X},\mathbf{Y}$.

For $a,b=1,\ldots, k+\ell$, denote by  $\mathbf{E}_{a,b}$ the elementary $(k+\ell)\times (k+\ell)$ matrix with 1 in the $(a,b)$-entry and  zero in all other entries.  Let $\epsilon_i: \mathfrak{gl}(k|\ell)\rightarrow \mathbb{C}$ be the linear function on $\mathfrak{gl}(k|\ell)$ defined by \begin{equation*}
  \epsilon_i(\mathbf{E}_{a,b})=\delta_{i,a}\delta_{a,b} \text{ for }i, a,b\in [1, k+\ell].
\end{equation*}
The free abelian group $P=\bigoplus\limits_{i=1}^{k+\ell}\mathbb{Z}\epsilon_i$ (resp. $P^{\vee}=\bigoplus\limits_{i=1}^{k+\ell}\mathbb{Z}\mathbf{E}_{b,b}$) is called the \textit{weight lattice} (resp. \textit{dual weight lattice}) of $\mathfrak{gl}(k|\ell)$, and there is a symmetric bilinear form
$(\,,\,)$ on $\mathfrak{h}^*=\mathbb{C}\otimes_{\mathbb{Z}}P$ defined by
\begin{equation*}
  (\epsilon_i,\epsilon_j)=(-1)^{\overline{i}}\delta_{i,j} \text{ for }i,j\in [1, k+\ell].
\end{equation*}
Then the simple roots of $\mathfrak{gl}(k,\ell)$ are $\alpha_i=\epsilon_i-\epsilon_{i+1}$, $i=1, \ldots, k+\ell-1$.
We have positive root system $\Phi^{+}=\{\alpha_{i,j}=\epsilon_i-\epsilon_j|1\leq i<j\leq k+\ell\}$ and negative root system $\Phi^{-}=-\Phi^{+}$. Define $\overline{\alpha}_{i,j}=\overline{i}+\overline{j}$ and call $\alpha_{i,j}$ is an even (resp. odd) root if $\overline{\alpha}_{i,j}=\overline{0}$ (resp. $\overline{1}$). Note that $\alpha_{k}$ is the only odd simple root. Denote by $\langle\cdot,\cdot\rangle$ the natural pairing between $P$ and $P^{\vee}$. Then the simple coroot $\alpha^{\vee}_i$ corresponding to $\alpha_i$ is the unique element in $P^{\vee}$ satisfying \begin{equation*}\langle \alpha^{\vee}_i,\lambda\rangle=(-1)^{\overline{i}}(\alpha_i, \lambda)\text{ for all }\lambda\in P. \end{equation*}
\end{point}
\begin{definition}The \textit{quantum superalgebra} $U_q(\mathfrak{gl}(k|\ell)$, that is, the \textit{quantized universal enveloping algebra} of $\fgl(k|\ell)$ is the unitary superalgebra over $\mathbb{K}$ generated by the homogeneous  elements
\begin{equation*}
  E_1, \ldots, E_{k+\ell-1}, F_1, \ldots, F_{k+\ell-1}, K_1^{\pm1}, \ldots, K_{k+\ell}^{\pm1}
\end{equation*}
with a $\mathbb{Z}_2$-gradation by letting $\overline{E}_k=\overline{F}_k=\overline{1}$,
 $\overline{E}_a=\overline{F}_a=\overline{0}$ for $a\neq k$, and $\overline{{K_i}^{\pm1}}=\overline{0}$. These generators satisfy the following relations:
\begin{enumerate}
  \item[(Q1)] $K_aK_b=K_bK_a, K_aK_a^{-1}=K_a^{-1}K_a=1$;
  \item[(Q2)] $K_aE_b=q^{\la\alpha^{\vee}_a,\alpha_b \ra}E_bK_a$;
\item[(Q3)] $E_aE_b=E_bE_a, F_aF_b=F_bF_a$  if $|a-b|>1$;
\item[(Q4)]$[E_a,F_b]=\delta_{a,b}\frac{\widetilde{K}_a-\widetilde{K}_a^{-1}}{q_a-q_a^{-1}}$, where $q_a=q^{(-1)^{\overline{a}}}$ and $\widetilde{K}_a=K_aK^{-1}_{a+1}$;
\item[(Q5)] For $a\neq k$ and $|a-b|>1$,
\begin{eqnarray*}
 && E_a^2E_b-(q_a+q_a^{-1})E_aE_bE_a+E_bE_a^2=0,\\
&&F_a^2F_b-(q_a+q_a^{-1})F_aF_bF_a+F_bF_a^2=0;
\end{eqnarray*}
\item[(Q6)] $E_k^2=F^2_k=0$,\\
$E_k\!\left(\!E_{k\!-\!1}E_kE_{k\!+\!1}\!\!+\!\!E_{k\!+\!1}E_{k}E_{k\!-\!1}\!\right)\!\!-
\!\!\left(\!q\!\!+\!\!q^{-\!1}\!\right)\!E_kE_{k\!-\!1}E_{k\!+\!1}E_k
\!\!+\!\!\left(\!E_{k\!-\!1}E_kE_{k\!+\!1}\!\!+\!\!E_{k\!+\!1}E_kE_{k\!-\!1}\!\right)\!E_k=0$,\\
$F_k\!\left(\!F_{k\!-\!1}F_kF_{k\!+\!1}\!\!+\!\!F_{k\!+\!1}F_{k}F_{k\!-\!1}\!\right)\!\!-
\!\!\left(\!q\!\!+\!\!q^{-\!1}\!\right)\!F_kF_{k\!-\!1}F_{k\!+\!1}F_k
\!\!+\!\!\left(\!F_{k\!-\!1}F_kF_{k\!+\!1}\!\!+\!\!F_{k\!+\!1}F_kF_{k\!-\!1}\!\right)\!F_k=0$.
\end{enumerate}
  \end{definition}
It is known that $U_q(\mathfrak{gl}(k|\ell))$  is a Hopf superalgebra with comultiplication $\Delta$ defined by
\begin{eqnarray*}
  &&\Delta(K_i^{\pm1})=K_i^{\pm 1}\otimes K_i^{\pm 1},\\
 &&\Delta(E_i)=E_i\otimes \widetilde{K}_i+1\otimes E_i, \\
&& \Delta(F_i)=F_i\otimes 1+ \widetilde{K}_i^{-1}\otimes F_i.
  \end{eqnarray*}

\begin{point}{}*\label{Sub-Sec:Vector-reps}
Let $V$ be a superspace over $\mathbb{K}$ with $\dim V= k|\ell$, that is, $V=\mathbb{C}^{k|\ell}\otimes_{\mathbb{C}}\mathbb{K}$, and let  $\mathfrak{B}=\{v_1,\ldots, v_{k+\ell}\}$ be its homogeneous basis with $\bar{v_i}=\bar{i}$ for $1\leq i\leq k+\ell$.  The \textit{vector representation} $\Psi$ of $U_q(\mathfrak{gl}(k|\ell))$ on $V$ is defined by
\begin{eqnarray*}
  &&\Psi(E_i)v_{j}=\left\{
                     \begin{array}{ll}
                       (-1)^{\overline{v}_j}v_{j-1}, &\quad \hbox{ if }j=i+1; \\
                       0, & \quad\hbox{ others.}
                     \end{array}
                   \right.;\\
&&\Psi(F_i)v_{j}=\left\{
                     \begin{array}{lll}
                       (-1)^{\overline{v}_j}v_{j+1}, &\quad \hbox{ if }j=i; \\
                       0, & \quad\hbox{ others.}
                     \end{array}
                   \right.\\
&&\Psi(K_i^{\pm 1})(v_j)=\left\{
                     \begin{array}{ll}
                       (-1)^{\overline{v}_j}q^{\pm 1}v_{j}, & \quad\hbox{ if }j=i; \\
                       0, & \quad\hbox{others.}
                     \end{array}
                   \right.
\end{eqnarray*}

For a positive integer $n$, we can define inductively a superalgebra homomorphism
\begin{equation*}
  \Delta^{(n)}: U_q(\mathfrak{gl}(k|\ell))\rightarrow U_q(\mathfrak{gl}(k|\ell))^{\otimes n},\quad \Delta^{(n)}=(\Delta^{(n-1)}\otimes \mathrm{id})\circ \Delta
\end{equation*}
 for each $n\geq 3$, where $\Delta^{(2)}=\Delta$. Therefore, $\Psi$ can be extended to the representation on tensor space $V^{\otimes n}$ via the Hopf superalgebra structure of $U_q(\mathfrak{gl}(k|\ell))$ for each $n$, we denote it by $\Psi^{\otimes n}$. More precisely, the $U_q(\mathfrak{gl}(k|\ell))$-act on $V^{\otimes n}$ is defined as follows:
\begin{eqnarray*}
  &&\Psi^{\otimes n}(E_a)=\sum_{p=0}^{n-1}\widetilde{K}_a^{\otimes p}\otimes \Psi(E_a)\otimes \mathrm{Id}^{\otimes^{n-1-p}},\\
&&\Psi^{\otimes n}(F_a)=\sum_{p=0}^{n-1}\mathrm{Id}^{\otimes p}\otimes \Psi(F_a)\otimes (\widetilde{K}_a^{-1})^{\otimes^{n-1-p}},\\
&&\Psi^{\otimes n}(K_{a})=K_a\otimes\cdots\otimes K_a.
\end{eqnarray*}
According to \cite[Proposition~3.1]{BKK}, the vector representation is an irreducible highest weight module $V(\epsilon_1)$ with highest weight $\epsilon_1$ and $V^{\otimes n}$ is complete reducible for all $n$.\end{point}

\begin{point}{}*
Now assume that $V=V^{(1)}\oplus\cdots\oplus V^{(m)}$, where $V^{(i)}$ is a subsuperspace of $V$ with $\dim V^{(i)}=k_i|\ell_i$ and homogeneous basis
\begin{equation*}
  \mathfrak{B}^{(i)}=\left\{v^{(i)}_1, \ldots, v_{k_i+\ell_i}^{(i)}\right\}, \quad 1\leq i\leq m,
\end{equation*}
where $v^{(i)}_1, \ldots, v^{(i)}_{k_i}$ is even and $v^{(i)}_{k_i+1}, \ldots, v^{(i)}_{k_i+\ell_i}$ is odd  for $i=1, \ldots, m$. In this way ,we obtain that $\mathfrak{B}=\mathfrak{B}^{(1)}\sqcup\cdots\sqcup\mathfrak{B}^{(m)}$ and the vectors in $\mathfrak{B}^{(i)}$ are said to be of \textit{color} $i$. Further  we linearly order the vectors $v_1^{(1)}, \ldots, v_{m}^{k_m+\ell_m}$ by the rule\begin{eqnarray*}
  v_{a}^{(i)}<v_{b}^{(j)}&\text{ if and only if }& i<j\text{ or }i=j\text{ and }a<b.                                                          \end{eqnarray*}We may identify the vectors  $v_1^{(1)}$, $\ldots$, $v^{(m)}_{k_m+\ell_m}$ with the vectors $v_1$, $\cdots$, $v_{k+\ell}$ as follows:
\begin{equation*}\label{Equ:Basis-index}
  \begin{array}{ccccccccccccc}
  v_{1}^{(1)}& \cdots & v_{k_1}^{(1)}& v_{k_1\negmedspace+\negmedspace1}^{(1)}&\cdots& v_{k_1\negmedspace+\negmedspace\ell_1}^{(1)}&\cdots&
  v_{1}^{(m)}& \cdots & v_{k_m}^{(m)}& v_{k_m+1}^{(m)}&\cdots& v_{k_m+\ell_m}^{(m)} \\
  \updownarrow& \vdots & \updownarrow & \updownarrow & \vdots & \updownarrow&\vdots&\updownarrow&\vdots&\updownarrow&\updownarrow&\vdots&\updownarrow \\
  v_1& \cdots & v_{k_1} & v_{k+1} & \cdots& v_{k\negmedspace+\negmedspace\ell_1}&\cdots&v_{k\negmedspace-\negmedspace k_m\negmedspace+\negmedspace1}&\cdots&v_k&v_{d_m\negmedspace-\negmedspace\ell_m\negmedspace+
  \negmedspace1}&\cdots&v_{k+\ell},
\end{array}
\end{equation*}

Let $\mathcal{I}(k,\ell;n)=\{\bi=(i_1, \ldots, i_n)|1\leq i_t\leq k+\ell, 1\leq t\leq n\}$.
For $\bi=(i_1, \ldots, i_n)\in \mathcal{I}(k,\ell;n)$, we write $v_{\bi}=v_{i_1}\otimes \cdots\otimes v_{i_n}$ and put $c_a(v_{\bi})=b$ if $v_{i_a}$ is of color $b$. Then $\mathfrak{B}^{\otimes n}=\{v_{\bi}|\bi\in \ci(k,\ell;n)\}$ is a homogeneous basis of $V^{\otimes n}$. We may and will identify $\mathfrak{B}^{\otimes n}$ with $\ci(k,\ell;n)$, that is, we will write $v_{\bi}$ by $\bi$, $\overline{v}_i$ by $\overline{i}$, $c_a(v_{\bi})$ by $c_a(\bi)$, etc., if there are no confusions. Clearly, $\overline{\bi}=\overline{i}_1+\cdots+\bar{i}_n$.

Clearly, the Lie superalgebra $\fgl(k_i|\ell_i)$ can be viewed as a subalgebra of $\fgl(k|\ell)$ for all $i=1, \ldots, m$. Therefore the Lie superalgebra $\mathfrak{g}=\mathfrak{gl}(k_1|\ell_1)\oplus\cdots\oplus\mathfrak{gl}(k_m|\ell_m)$ is a subalgebra of $\fgl(k|\ell)$ and its quantum superalgebra $U_q(\mathfrak{g})$ can be naturally embedded in  $U_q(\mathfrak{gl}(k,\ell))$ as a $\mathbb{K}$-subalgebra generated by
\begin{equation}\label{Equ:Uq(g)-generators}
 \mathscr{G}=\left\{ E_a, F_a, K_{b}^{\pm1}\mid a\in \{1,2, \ldots, d_m\}\backslash\{d_1, d_2, \ldots,d_m\}, 1\leq b\leq d_m\right\}.
\end{equation}
Hence  the restriction of $U_q(\fgl(k|\ell))$-representation $(\Psi^{\otimes n}, V^{\otimes n})$ gives a $U_q(\mathfrak{g})$-representation, we denote it by $(\Psi^{\otimes n}, V^{\otimes n})$.
\end{point}

\begin{point}{}*\label{Sec:Dominant-order}
 Recall that a composition (resp. partition) $\gl=(\gl_1, \gl_2, \ldots)$ of $n$, denote $\gl\models n$ (resp. $\gl\vdash n$), is a sequence (resp. weakly decreasing sequence) of  nonnegative integers such that $|\gl|=\sum_{i\geq1}\gl_i=n$ and write $\ell(\gl)$ the length of $\gl$, i.e., the number of nonzero parts of $\gl$. A {\it multipartition} of $n$ is an ordered $m$-tuple $\bgl=(\gl^{(1)}; \ldots; \gl^{(m)})$ of partitions $\lambda^{i}$ such that
$n=\sum_{i=1}^m|\lambda^{i}|$. We denote by $\mpn$ the set of all multipartitions of $n$. Then $\mpn$ is a poset under dominance $\unrhd$, where
$$
\bgl\unrhd\bmu\Longleftrightarrow\displaystyle\sum_{k=1}^{i-1}|\gl^{k}|+\sum_{\ell=1}^j\gl_l^{i}\geq
\sum_{k=\ell}^{i-1}|\mu^{k}|+\sum_{\ell=1}^j\mu_\ell^{i}\quad
\text{ for all }1\leq i\leq m \text{ and }j\geq 1.$$
 We write $\bgl\rhd\bmu$
if $\bgl\unrhd\bmu$ and $\bgl\neq \bmu$.

A partition $\gl=(\gl_1, \gl_2, \cdots)\vdash n$ is said to be a \textit{$(k, \ell)$-hook partition} of $n$ if $\gl_{k+1}\leq \ell$. We let  $H(k,\ell;n)$ denote the set of all $(k,\ell)$-hook partitions of $n$, that is
\begin{eqnarray*}\label{Def:hook}
H(k,\ell;n)=\{\gl=(\gl_1,\gl_2,\cdots)\vdash n\mid \gl_{k+1}\le \ell\}.
\end{eqnarray*}
A multipartition $\bgl=(\gl^{(1)}; \ldots; \gl^{(m)})$ of $n$ is said to be a \textit{$(\bk,\bl)$-hook multipartition} of $n$ if $\gl^{(i)}$ is a $(k_i,\ell_i)$-hook partition for all $i=1, \ldots,m$. We denote by $H(\bk|\bl; m,n)$ the set of all  $(\bk,\bl)$-hook multipartitions of $n$. Thanks to \cite[Theorem~2]{Serg} and  \cite[Theorem~3.20]{B-Regev}, the irreducible representations of $U_q(\mathfrak{gl}(k,\ell))$ occurring in $V^{\otimes n}$ are parameterized by the $(k,\ell)$-hook partitions of $n$. Note that $U_q(\mathfrak{g})=U_q(\mathfrak{gl}(k_1,\ell_1))\otimes\cdots\otimes U_q(\mathfrak{gl}(k_m,\ell_m))$. As a consequence,  the irreducible representations of $U_q(\mathfrak{g})$ occurring in $V^{\otimes n}$ are parameterized by the $(\bk,\bl)$-hook multipartitions of $n$.
\end{point}

\begin{point}{}*
The {\it diagram} of a multipartition $\bgl$ is the set
$$ [\bgl]:=\{(i,j,c)\in\bz_{>0}\times\bz_{>0}\times \mathbf{m}|1\le j\le\lambda^c_i\}, \quad\text{ where }\mathbf{m}=\{1, \dots, m\}.$$
The elements of $[\bgl]$ are the \textit{nodes} of $\bgl$. By a {\it $\bgl$-tableau}, we mean a bijection $\ft: [\bgl]\rightarrow\{1,2,\dots\}$ and write $\text{Shape}(\ft)=\bgl$ if $\ft$ is a $\bgl$-tableau. If its entries are from the set $\{1,2,\ldots,n\}$ then it is called
an \textit{$n$-tableau}. Of course an  $n$-tableau is also an $n+1$-tableau, etc.
We may and will identify a tableau $\ft$ with an $m$-tuple of tableaux
$\ft=(\ft^1; \dots; \ft^m)$, where $\ft^{c}$ is a $\lambda^{c}$-tableau,  $c=1, \cdots, m$, which is called the {\it$c$-component} of $\ft$.
A tableau is {\it (semi) standard} if in each component the entries (weakly) increase along the rows and
strictly down along the columns and  denote by $\std(\bgl)$ the set of all standard
$\bgl$-tableaux. Given a standard tableau $\ft$ and  an integer $i$, we define the {\it residue} of $i$ in $\ft$ to be $\res_\ft(i)=Q_cq^{2(b-a)}$ if $i$
appears in the node $(a, b,c)$ of $\ft$.

Let $\bar{\mathbf{0}}=\{0_1, \cdots, 0_k\}$ and $\bar{\mathbf{1}}=\{1_1, \cdots, 1_{\ell}\}$ with $0_1<\cdots<0_k<1_1<\cdots<1_{\ell}$. Then a tableau $\ft$ of shape $\gl\vdash n$ is said to be $(k,\ell)$-\textit{semistandard} if
\begin{enumerate}
  \item[(i)] the $\bar{\mathbf{0}}$ part (i.e. the boxes filled with entries $0_i$'s) of $\ft$ is a tableau,
  \item[(ii)]the $0_i$'s are nondecreasing in row, strictly increasing in columns,

   \item[(ii)]the $1_i$'s are nondecreasing in columns, strictly increasing in rows.
\end{enumerate}
\end{point}

\section{The sign $q$-permutation representation}\label{Sec:Sign-q-reps}
This section devotes to  introduce an $\h$-action on $V^{\otimes n}$ and  prove that it is a (super) representation of $\h$ by adapting  the ideas of \cite{SS,Moon,Mit}.

\begin{point}{}*Let $W_{m,n}$ be the complex reflection group of type $G(m,1,n)$. According to \cite{shephard-toda}, $W_{m,n}$ has a presentation with generators
  $s_0, s_1, \dots, s_{n-1}$ where the defining relations are $s_0^m=1, s_1^2=\cdots=s_{n-1}^2=1$ and the homogeneous relations
  \begin{align*}&s_0 s_1s_0 s_1=s_1s_0 s_1s_0,&&\\
 & s_is_j=s_js_i,&&\text{ if } |i-j|>1,\\
&s_is_{i+1}s_i=s_{i+1}s_{i}s_{i+1},&& \text{ for }1\leq i\leq n-2.\end{align*}
 It is well-known that $W_{m,n}\cong(\mathbb{Z}/m\bz)^{n}\rtimes \mathfrak{S}_{n}$, where $s_1, \dots, s_{n-1}$ are generators of the symmetric group $\mathfrak{S}_{n}$ of degree $n$ corresponding to transpositions $(1\,2)$, $\ldots$, $(n\!-\!1\,n)$.

For $a=1, \ldots, n-1$ and
$\bi=(i_1, \ldots, i_a,i_{a+1},\ldots, i_n)$, we define the following right action
\begin{equation*}\label{Equ:is_a}
  \bi s_a:=(i_1, \ldots, i_{a-1}, i_{a+1}, i_a, i_{a+2}, \ldots, i_n).
\end{equation*}

Following Sergeev \cite[\S\,1.1]{Serg}  or Berele-Regev \cite[Definition~1.9]{B-Regev}, there is a right action $\phi$ of $\mathbb{C}\mathfrak{S}_n$ on $V^{\otimes n}$ defined on generators by
\begin{eqnarray}\label{Equ:W_mn-sign}
  s_a(\bi)&:=&\left\{\begin{array}{ll}\vspace{1\jot}
 (-1)^{\overline{i}_a}\bi,& \text{if }i_a=i_{a+1};\\
(-1)^{\overline{i}_a\overline{i}_{a+1}}\bi s_a,& \text{if }i_a\neq i_{a+1}.
                         \end{array}\right.
\end{eqnarray}
\end{point}

\begin{point}{}* The \textit{Ariki-Koike algebra} \cite{AK} or the \textit{cyclotomic Hecke algebra} $\h$ associated to $W_{m,n}$ \cite{BM:cyc}, is the unital
associative $\mathbb{K}$-algebra  generated by
$g_0,g_1,\dots,g_{n-1}$ and subject to relations
\begin{align*}&(g_0-Q_1)\dots(g_0-Q_m)=0,&&\\
&g_0g_1g_0g_1=g_1g_0g_1g_0,&&\\
&g_i^2=(q-q^{-1})g_i+1, &&\text{ for }1\leq i<n,\\
&g_ig_j=g_jg_i, &&\text{ for }|i-j|>2,\\
 &g_ig_{i+1}g_i=g_{i+1}g_{i}g_{i+1}, &&\text{ for }1\leq i<n-1.\end{align*}

Let $w\in \mathfrak{S}_n$ and let $s_{i_1}s_{i_2}\cdots s_{i_k}$ be a reduced expression for $w$. Then $g_{w}:=g_{i_1}g_{i_2}\cdots g_{i_k}$ is independent of the choice of reduced expression and  $\{g_{w}|w\in \mathfrak{S}_n\}$ is a linear basis of the subalgebra $\mathcal{H}_n(q)$ of $\h$ generated by $g_1, \ldots, g_{n-1}$, that is, $\mathcal{H}_n(q)$ is the Iwahori-Hecke algebra associated to  $\mathfrak{S}_n$.
\end{point}

For $a=1, \ldots, n-1$, we define the endomorphisms $T_a, S_a\in \mathrm{End}_\mathbb{K}(V^{\otimes n})$ as follows:
\begin{equation}\label{Equ:Ta-action}
  T_a(\bi):=\left\{\begin{array}{ll}\vspace{2\jot}(q-q^{-1})\bi
 +(-1)^{\overline{i}_a\overline{i}_{a+1}}\bi s_a,& \text{if }i_a<i_{a+1};\\ \vspace{2\jot}
 \frac{(q-q^{-1})+(-1)^{\overline{i}_a}(q+q^{-1})}{2} \bi,& \text{if }i_a=i_{a+1}; \\
  (-1)^{\overline{i}_a\overline{i}_{a+1}}\bi s_a,&\text{if } i_a>i_{a+1}.
                         \end{array}\right.
\end{equation}\begin{eqnarray}
\label{Equ:S_a}
  S_a(\bi)&:=&\left\{\begin{array}{ll}
 T_a(\bi), & \hbox{ if } c_a(\bi)= c_{a+1}(\bi);\\
 s_a(\bi), & \hbox{ if } c_a(\bi)\neq c_{a+1}(\bi).\end{array}\right.
\end{eqnarray}

The following easy verified facts will be used latter.

\begin{lemma}\label{Lemm:Ta-new}
For all $\bi\in\ci(k,\ell;n)$ and $1\leq a<n$, we have \begin{enumerate}
                \item  $
  T_a(\bi)=\left\{\begin{array}{ll}\vspace{2\jot}(q-q^{-1})\bi
 +s_a(\bi),& \text{if }i_a<i_{a+1};\\ \vspace{2\jot}
 \frac{q-q^{-1}}{2}\bi+\frac{q+q^{-1}}{2}s_a(\bi),& \text{if }i_a=i_{a+1}; \\
  s_a(\bi),&\text{if } i_a>i_{a+1}.
                         \end{array}\right.$

                \item  $T_a$ is invertible and $
  T_a^{-1}(\bi):=\left\{\begin{array}{ll}\vspace{2\jot}
 s_a(\bi),& \text{if }i_a<i_{a+1};\\ \vspace{2\jot}
 -\frac{q-q^{-1}}{2}\bi+\frac{q+q^{-1}}{2}s_a(\bi),& \text{if }i_a=i_{a+1}; \\
 (q-q^{-1})\bi+ s_a(\bi),&\text{if } i_a>i_{a+1}.
                         \end{array}\right.$
              \end{enumerate}
\end{lemma}
\begin{proof}(i) follows directly by applying Eq.~(\ref{Equ:W_mn-sign}).  Since $T^2_a=(q-q^{-1})T_a+1$, $T_a^{-1}=T_a-(q-q^{-1})$. Thus (ii) follows directly by applying (i).
\end{proof}

Now let $S_0(\bi):=Q_{c_1(\bi)}\bi$ and $\theta=S_{n-1}\cdots S_{1}$. We define
 $T_0\in \mathrm{End}_{\mathbb{K}}(V^{\otimes n})$ as following
\begin{eqnarray}\label{Equ:T_0-action}
T_0(\bi):&=&T_{1}^{-1}\cdots T_{n-1}^{-1}\theta S_0(\bi).
\end{eqnarray}

Thanks to \cite{Moon,Mit}, Eq.~(\ref{Equ:Ta-action}) defines a (super) representation of $\mathcal{H}_n(q)$. The remainder of this section devotes to show that Eqs.~(\ref{Equ:Ta-action}) and (\ref{Equ:T_0-action}) define a (super) representation of $\h$.

\begin{lemma}\label{Lemm:T0-relation} For $j,p\geq 1$, we denote by $V_{j,p}$ the subspace of $V^{\otimes n}$ spanned by basis elements $\bi$ such that $c_p(\bi)\geq j$. If $\bi\in V_{j,p}$ then $T_{p}^{-1}\cdots T_{n-1}^{-1}S_{n-1}\cdots S_{p}(\bi)\in \bi+V_{j+1,p}$.
 \end{lemma}

\begin{proof}
 We use the backward induction on $p$ to prove the claim. Note that for all $p=1, \ldots, n-1$, we have
 \begin{eqnarray}\label{Equ:v-ST}
    &&T_{p}^{-1}S_{p}(\bi)=\left\{\begin{array}{ll}
   \bi+(q-q^{-1})s_{p}(\bi), & \hbox{ if }c_{p}(\bi)> c_{p+1}(\bi);\\
  \bi, & \hbox{ others}.\end{array}\right.
  \end{eqnarray}
In particular, the lemma holds for $p=n-1$.
  Now assume that for all $p$ and $\bi'\in V_{j,p}$,
  \begin{equation*}
 T_{p}^{-1}\cdots T_{n-1}^{-1}S_{n-1}\cdots S_{p}(\bi')\in
  \bi'+V_{j+1,p}.
\end{equation*}
Thanks to Lemma~\ref{Lemm:Ta-new}(i),  $T_{p-1}(V_{j,p-1})=V_{j,p}$ for all $j\geq 1$, which implies $S_{p-1}(V_{j,p-1})\in V_{j,p}$ due to Eq.~(\ref{Equ:S_a}).

  For any $\bi\in V_{j,p-1}$, the induction argument shows
   \begin{eqnarray*}
    (T_{p-1}^{-1}\cdots T_{n-1}^{-1}S_{n-1}\cdots S_{p-1})(\bi)&=& T_{p-1}^{-1}(T_p^{-1}\cdots T_{n-1}^{-1}S_{n-1}\cdots S_p)(S_{p-1}(\bi))\\ &\in&T_{p-1}^{-1}(S_{p-1}(\bi)+V_{j+1,p})\\
    &=&T_{p-1}^{-1}S_{p-1}(\bi)+V_{j+1,p}\\
    &\in& \bi+V_{j+1,p},
         \end{eqnarray*}
  where the last inclusion follows by Eq.~(\ref{Equ:v-ST}). The lemma is proved.\end{proof}

We will need the following facts.
\begin{lemma}\label{Lemm:TjSj} For all $j\geq 2$, we have the following facts:
\begin{equation*}
   S_jS_{j-1}T_j=T_{j-1}S_jS_{j-1},
\end{equation*}
\begin{equation}\label{Equ:TjSj}
 S_jS_{j-1}S_j S_{j-1}^{-1}T_{j-1}=T_jS_jS_{j-1}S_jS_{j-1}^{-1},  \\
\end{equation}
\begin{equation*}
 S_jS_{j-1}S_j S_{j-1}T_{j-1}=T_jS_jS_{j-1}S_jS_{j-1}.
\end{equation*}
\end{lemma}

\begin{proof}
Let $q_+=q+q^{-1}$ and $q_{*}=q-q^{-1}$. Without loss of generality, we may assume that $j=2$ and $\bi=(i_1, i_2, i_3)$.  Therefore we have the following five cases:

\begin{enumerate}\setlength{\itemsep}{1\jot}
\item[(1)] If $c_1(\bi)=c_2(\bi)=c_{3}(\bi)$ then $S_1=T_1$, $S_2=T_2$. Furthermore Eq.~(\ref{Equ:TjSj}) follows owing to \cite[Proposition~2.9]{Moon}.

\item[(2)]If $c_1(\bi)$, $c_2(\bi)$ and $c_3(\bi)$ are pairwise different then $S_1(\bi)= s_1(\bi)$, $S_2(\bi)=s_2(\bi)$. Thus we only need to consider the following cases: (a) $i_1<i_2<i_3$; (b) $i_1<i_2>i_3$; (c) $i_1>i_2>i_3$. Apply  Lemma~\ref{Lemm:Ta-new}(i) and Eq.~(\ref{Equ:S_a}), we obtain that

\vspace{1\jot}
\begin{enumerate}\setlength{\itemsep}{1\jot}
  \item[(a)]  $\begin{array}{lll}&&S_2S_1T_2(\bi)=(q-q^{-1})s_2s_1(\bi)+s_1s_2s_1(\bi)=T_1S_2S_1(\bi);\\
&&S_2S_1S_2 S_1^{-1}T_1(\bi)=(q-q^{-1})s_1s_2(\bi)+s_1s_2s_1(\bi)=T_2S_2S_1S_2S_1^{-1}(\bi);\\
&&S_2S_1S_2 S_1T_1(\bi)=S_2S_1S_2 S_1^{-1}T_1(\bi)=T_2S_2S_1S_2S_1^{-1}(\bi)=T_2S_2S_1S_2S_1(\bi);
\end{array}$

  \item[(b)]$\begin{array}{lll}
&&S_2S_1T_2(\bi)=s_1s_2s_1(\bi)=T_1S_2S_1(\bi);\\
&&S_2S_1S_2 S_1^{-1}T_1(\bi)=(q-q^{-1})s_1s_2(\bi)+s_1s_2s_1(\bi)=T_2S_2S_1S_2S_1^{-1}(\bi);\\
&&S_2S_1S_2 S_1T_1(\bi)=S_2S_1S_2 S_1^{-1}T_1(\bi)=T_2S_2S_1S_2S_1^{-1}=T_2S_2S_1S_2S_1(\bi);
\end{array}$

  \item[(c)]$\begin{array}{lll}&&S_2S_1T_2(\bi)=s_2s_1s_2(\bi)=s_1s_2s_1(\bi)=T_1S_2S_1(\bi);\\
&&S_2S_1S_2 S_1^{-1}T_1(\bi)=s_2s_1s_2(\bi)=T_2S_2S_1S_2S_1^{-1}(\bi);\\
&&S_2S_1S_2 S_1T_1(\bi)=s_1s_2T_1(\bi)=T_2s_1s_2(\bi)=T_2S_2S_1S_2S_1(\bi);
\end{array}$
\end{enumerate}
Therefore, in this case Eq.~(\ref{Equ:TjSj}) hold.

\item[(3)] If $c_1(\bi)=c_2(\bi)\neq c_3(\bi)$ then we only need to check the following six cases:

\begin{enumerate}
\item $i_1=i_2<i_3$:
\begin{eqnarray*}
S_2S_1T_2(\bi)&=&q_*S_2S_1(\bi)+S_2S_1s_2(\bi)\\
&=&q_*s_2T_1(\bi)+T_2s_1s_2(\bi)\\
&=&\frac{1}{2}q_+q_*s_2s_1(\bi)\!+\!\frac{1}{2}q_+s_1s_2s_1(\bi)\!-\!\frac{1}{2}q_*^2s_2(\bi)\!-\!
\frac{1}{2}q_*s_1s_2(\bi)\\
&=&\frac{1}{2}q_+T_1s_2s_1(\bi)-\frac{1}{2}q_*T_1s_2(\bi)\\
&=&T_1S_2T_1(\bi)\\
&=&T_1S_2S_1(\bi);                            \end{eqnarray*}
\begin{eqnarray*}
S_2S_1S_2S_1^{-1}T_1(\bi)&=&S_2S_1S_2(\bi)\\
&=&T_2s_1s_2(\bi)\\
&=&\frac{1}{2}q_+s_2s_1s_2(\bi)-\frac{1}{2}q_*s_1s_2(\bi),
\end{eqnarray*}
\begin{eqnarray*}
T_2S_2S_1S_2S_1^{-\!1}(\bi)&=&T_2S_2S_1S_2T_1^{-1}(\bi)\\
&=&\frac{1}{2}q_+T_2^2s_1s_2s_1(\bi)-\frac{1}{2}q_*T_2^2s_1s_2(\bi)\\
&=&\frac{q_*}{2}s_1s_2(\bi)\!+\!\frac{q_+}{2}s_1s_2s_1(\bi)
\!+\!\frac{q_*^2}{2}T_2s_1s_2(\bi)\!+\!\frac{q_+q_*}{2}T_2s_1s_2s_1(\bi)\\
&=&\frac{q_+}{2}s_2s_1s_2(\bi)-\frac{q_*}{2}s_1s_2(\bi);                            \end{eqnarray*}
\begin{eqnarray*}
S_2S_1S_2S_1T_1(\bi)&=&S_2S_1S_2T_1^2(\bi)\\
&=&S_2S_1S_2(\bi)+q_*S_2S_1S_2T_1(\bi)\\
&=&T_2s_1s_2(\bi)+q_*T_2s_1s_2T_1(\bi)\\
&=&s_1s_2T_1(\bi)+q_*T_2s_1s_2T_1(\bi)\\
&=&T_2^2s_1s_2T_1(\bi)\\
 &=&T_2S_2S_1S_2S_1(\bi);
\end{eqnarray*}

\item $i_1=i_2>i_3$:
\begin{eqnarray*}
S_2S_1T_2(\bi)&=&T_2s_1s_2(\bi)\\
&=&\frac{q_+}{2}s_1s_2s_1(\bi)-\frac{q_*}{2}s_1s_2(\bi)\\
&=&s_1s_2T_1(\bi)\\
&=&T_1S_2S_1(\bi);
\end{eqnarray*}
\begin{eqnarray*}
S_2S_1S_2S_1^{-1}T_1(\bi)&=&T_2s_1s_2(\bi)\\
&=&\frac{q_+}{2}s_2s_1s_2(\bi)-\frac{q_*}{2}s_1s_2(\bi),
\end{eqnarray*}
\begin{eqnarray*}
T_2S_2S_1S_2S_1^{-1}(\bi)&=&T_2^2s_1s_2T_1^{-1}(\bi)\\
&=&q_*T_2s_1s_2T_1^{-1}(\bi)+s_1s_2T_1^{-1}(\bi)\\
&=&\frac{q_*^2}{2}T_2s_1s_2(\bi)\!+\!\frac{q_+q_*}{2}T_2s_1s_2s_1(\bi)\!+\!
\frac{q_*}{2}s_1s_2(\bi)\!+\!\frac{q_+}{2}s_1s_2s_1(\bi)\\
&=&\frac{q_+}{2}s_2s_1s_2(\bi)-\frac{q_*}{2}s_1s_2(\bi);                            \end{eqnarray*}
\begin{eqnarray*}
S_2S_1S_2S_1T_1(\bi)&=&S_2S_1S_2T_1^2(\bi)\\
&=&T_2s_1s_2(\bi)+q_*T_2s_1s_2T_1(\bi)\\
&=&s_1s_2T_1(\bi)+q_*T_2s_1s_2T_1(\bi)\\
&=&T_2^2s_1s_2T_1(\bi)\\
&=&T_2S_2S_1S_2S_1(\bi);
\end{eqnarray*}
 \item $i_1<i_2<i_3$:

\begin{eqnarray*}
S_2S_1T_2(\bi)&=&q_*s_2T_1(\bi)+T_2s_1s_2(\bi)\\
&=&q_*^2s_2(\bi)+q_*(s_1s_2+s_2s_1)(\bi)+s_1s_2s_1(\bi)\\
&=&q_*T_1s_2(\bi)+T_1s_2s_1(\bi)\\
&=&T_1S_2S_1(\bi);
\end{eqnarray*}

\begin{eqnarray*}
T_2S_2S_1S_2S_1^{-1}(\bi)&=&T_2^2s_1s_2s_1(\bi)\\
&=&q_*T_2s_1s_2s_1(\bi)+s_1s_2s_1(\bi)\\
&=&q_*s_1s_2(\bi)+s_1s_2s_1(\bi)\\
&=&T_2s_1s_2(\bi)\\
&=&S_2S_1S_2(\bi)\\
&=& S_2S_1S_2S_1^{-1}T_1(\bi)
 \end{eqnarray*}
\begin{eqnarray*}
S_2S_1S_2S_1T_1(\bi)&=&S_2S_1S_2T_1^2(\bi)\\
&=&T_2s_1s_2(\bi)+q_*T_2s_1s_2T_1(\bi)\\
&=&s_1s_2T_1(\bi)+q_*T_2s_1s_2T_1(\bi)\\
&=&T_2^2s_1s_2T_1(\bi)\\
&=&T_2S_2S_1S_2S_1(\bi);
\end{eqnarray*}
\item $i_3<i_1<i_2$:
\begin{eqnarray*}
S_2S_1T_2(\bi)&=&q_*s_1s_2(\bi)+s_2s_1s_2(\bi)\\
&=&q_*T_1s_2(\bi)+T_1s_2s_1(\bi)\\
&=&T_1S_2S_1(\bi);
\end{eqnarray*}
\begin{eqnarray*}
S_2S_1S_2S_1^{-1}T_1(\bi)&=&T_2s_1s_2(\bi)\\
&=&q_*s_1s_2(\bi)+s_2s_1s_2(\bi)\\
&=&q_*T_2s_1s_2s_1(\bi)+s_1s_2s_1(\bi)\\
&=&T_2^2s_1s_2s_1(\bi)\\
&=&T_2S_2S_1S_2S_1^{-1}(\bi);
\end{eqnarray*}
\begin{eqnarray*}
S_2S_1S_2S_1T_1(\bi)&=&S_2S_1S_2T_1^2(\bi)\\
&=&T_2s_1s_2(\bi)+q_*T_2s_1s_2T_1(\bi)\\
&=&s_1s_2T_1(\bi)+q_*T_2s_1s_2T_1(\bi)\\
&=&T_2^2s_1s_2T_1(\bi)\\
&=&T_2S_2S_1S_2S_1(\bi).
\end{eqnarray*}
\item $i_2<i_1<i_3$:
\begin{eqnarray*}
S_2S_1T_2(\bi)&=&q_*s_1s_2+s_2s_1s_2\\
&=&q_*T_1s_2(\bi)+T_1s_2s_1(\bi)\\
&=&T_1S_2S_1(\bi);
\end{eqnarray*}
\begin{eqnarray*}
S_2S_1S_2S_1^{-1}T_1(\bi)&=&T_2s_1s_2(\bi)\\
&=&q_*s_1s_2(\bi)+s_2s_1s_2(\bi)\\
&=&q_*T_2s_1s_2s_1(\bi)+s_1s_2s_1(\bi)\\
&=&T_2^2s_1s_2s_1(\bi)\\
&=&T_2S_2S_1S_2S_1^{-1}(\bi);
\end{eqnarray*}
\begin{eqnarray*}
S_2S_1S_2S_1T_1(\bi)&=&S_2S_1S_2T_1^2(\bi)\\
&=&T_2s_1s_2(\bi)+q_*T_2s_1s_2T_1(\bi)\\
&=&s_1s_2T_1(\bi)+q_*T_2s_1s_2T_1(\bi)\\
&=&T_2^2s_1s_2T_1(\bi)\\
&=&T_2S_2S_1S_2S_1(\bi);
\end{eqnarray*}
\item $i_3<i_2<i_1$:
\begin{eqnarray*}
S_2S_1T_2(\bi)&=&q_*s_1s_2+s_2s_1s_2\\
&=&q_*T_1s_2(\bi)+T_1s_2s_1(\bi)\\
&=&T_1S_2S_1(\bi);
\end{eqnarray*}
\begin{eqnarray*}
S_2S_1S_2S_1^{-1}T_1(\bi)&=&T_2s_1s_2(\bi)\\
&=&q_*s_1s_2(\bi)+s_2s_1s_2(\bi)\\
&=&q_*T_2s_1s_2s_1(\bi)+s_1s_2s_1(\bi)\\
&=&T_2^2s_1s_2s_1(\bi)\\
&=&T_2S_2S_1S_2S_1^{-1}(\bi);
\end{eqnarray*}
\begin{eqnarray*}
S_2S_1S_2S_1T_1(\bi)&=&S_2S_1S_2T_1^2(\bi)\\
&=&T_2s_1s_2(\bi)+q_*T_2s_1s_2T_1(\bi)\\
&=&s_1s_2T_1(\bi)+q_*T_2s_1s_2T_1(\bi)\\
&=&T_2^2s_1s_2T_1(\bi)\\
&=&T_2S_2S_1S_2S_1(\bi).
\end{eqnarray*}
\end{enumerate}
Therefore  Eq.~(\ref{Equ:TjSj}) hold in this case.

\item[(4)]The remainder cases $c_1(\bi)\neq c_2(\bi)=c_3(\bi)$ and $c_1(\bi)=c_3(\bi)\neq c_2(\bi)$  can be verified by a similar way.
\end{enumerate}
As a consequence, we prove the lemma.
\end{proof}

Now we can prove the main result of this section.

\begin{theorem}\label{Them:Phi-reps}Keeping notation as above, then the $\mathbb{K}$-linear map $\Phi: \h\rightarrow \mathrm{End}_{\mathbb{K}}(V^{\otimes n})$ defined by $g_a\mapsto T_a$ ($a=0, \ldots, n-1$) is a (super) representation of $\h$.
\end{theorem}

\begin{proof}
  Thanks to  \cite[Proposition~2.9]{Moon} or \cite[Theorem~2.1]{Mit}, it suffices to show that the following three relations hold:
\begin{eqnarray}
 \label{Equ:T_0-relations} &&(T_0-Q_1)\cdots(T_0-Q_m)=0; \\
 \label{Equ:T_0T_1}  &&T_0T_1T_0T_1=T_1T_0T_1T_0;\\
 \label{Equ:T_0-T_i} &&T_0T_i=T_iT_0,  \hbox{ for }i\geq 2.
  \end{eqnarray}

 Applying Eq.~(\ref{Equ:T_0-action}) and Lemma~\ref{Lemm:T0-relation}, $T_0(\bi)=Q_{c_1(\bi)}\bi+V_{j+1,1}$ for  any $\bi\in V_{j,1}$. It follows that
   \begin{eqnarray*}(T_0-Q_j)(\bi)&=&(Q_{c_1(\bi)}-Q_j)\bi+V_{j+1,1}.
   \end{eqnarray*}
   Since $\bi\in V_{j,1}$, we have $c_1(\bi)\geq j$. Therefore $(Q_{c_1(\bi)}-Q_j)(\bi)=0$ if $c_1(\bi)=j$, and $(Q_{c_1(\bi)}-Q_j)(\bi)\in V_{j+1,1}$  if $c_1(\bi)>j$, that is, $(T_0-Q_j)(\bi)\in V_{j+1,1}$. Finally notice that $V_{1,1}=V^{\otimes n}$ and $V_{n+1,1}=\{0\}$. As a consequence, Eq.~(\ref{Equ:T_0-relations}) holds.

Now note that $S_0$ commutes with $T_2, \cdots, T_{n-1}$ and $S_iT_j=T_jS_i$ for $|i-j|>2$. Lemma~\ref{Lemm:TjSj} implies
\begin{equation}\label{Equ:theta-T}
  \theta T_j=T_{j-1}\theta,\qquad j=2, \ldots, n-1.
\end{equation}

Since $S_iS_j=S_jS_i$ for $|i-j|\geq 2$, we have
\begin{eqnarray*}
  \theta ^2T_1 &=& (S_{n-1}\cdots S_{1})(S_{n-1}\cdots S_{1})T_1\\&=&S_{n-1}(S_{n-2}S_{n-1})\cdots(S_2S_3)(S_1S_2)S_1T_1\\
&=&(S_{n-1}S_{n-2}S_{n-1}S_{n-2}^{-1})(S_{n-2}S_{n-3}S_{n-2}S_{n-3}^{-1})
\cdots(S_3S_2S_3S_{2}^{-1})(S_2S_1S_2S_1)T_1\\
&=&(S_{n-1}S_{n-2}S_{n-1}S_{n-2}^{-1})(S_{n-2}S_{n-3}S_{n-2}S_{n-3}^{-1})
\cdots(S_3S_2S_3S_{2}^{-1})T_2(S_2S_1S_2S_1)\\
&=&T_{n-1}(S_{n-1}S_{n-2}S_{n-1}S_{n-2}^{-1})(S_{n-2}S_{n-3}S_{n-2}S_{n-3}^{-1})
\cdots(S_3S_2S_3S_{2}^{-1})(S_2S_1S_2S_1)\\
&=&T_{n-1}\theta^2.
\end{eqnarray*}

Now we show $S_1^{-1}S_0S_1S_0T_1=T_1S_1^{-1}S_0S_1 S_0$. To do this, we may assume  $\bi=(i_1,i_2)$. According to Eq.~(\ref{Equ:Ta-action}),
\begin{eqnarray*}
  S_1^{-1}S_0S_1S_0T_1(\bi) &=& \left\{
                              \begin{array}{ll}
                                Q_{c_1(\bi)}^2T_1(\bi), & \hbox{ if }c_1(\bi)=c_2(\bi); \\
                                Q_{c_1(\bi)}Q_{c_2(\bi)}T_1(\bi), & \hbox{ if }c_1(\bi)\neq c_2(\bi).                              \end{array}                           \right.\\
  T_1S_1^{-1}S_0S_1S_0(\bi) &=& \left\{
                              \begin{array}{ll}
                                Q_{c_1(\bi)}^2T_1(\bi), & \hbox{ if }c_1(\bi)=c_2(\bi); \\
                                Q_{c_1(\bi)}Q_{c_2(\bi)}T_1(\bi), & \hbox{ if }c_1(\bi)\neq c_2(\bi).
                              \end{array}
                            \right.
\end{eqnarray*}
Combing the above two equalities, we get $(\theta S_0)^2T_1=T_{n-1}(\theta S_0)^2$.

 As a consequence, we yield that
\begin{eqnarray*}
  T_0T_1T_0T_1&=&(T_{1}^{-1}\cdots T_{n-1}^{-1})\theta S_0(T_{2}^{-1}\cdots T_{n-1}^{-1})\theta S_0T_1\\
&=&(T_{1}^{-1}\cdots T_{n-1}^{-1})(T_{1}^{-1}\cdots T_{n-2}^{-1})(\theta S_0)^2T_1\\
&=&(T_{1}^{-1}\cdots T_{n-1}^{-1})(T_{1}^{-1}\cdots T_{n-2}^{-1})T_{n-1}(\theta S_0)^2;\\
T_1T_0T_1T_0&=&(T_{2}^{-1}\cdots T_{n-1}^{-1})\theta S_0(T_{2}^{-1}\cdots T_{n-1}^{-1})\theta S_0\\
&=&(T_{2}^{-1}\cdots T_{n-1}^{-1})(T_{1}^{-1}\cdots T_{n-2}^{-1})(\theta S_0)^2\\
&=&(T_{2}^{-1}\cdots T_{n-1}^{-1})(T_{1}^{-1}\cdots T_{n-2}^{-1})(\theta S_0)^2.
\end{eqnarray*}
Thanks to \cite[Lemma~2.3(4)]{ATY}),  we yield that $T_0T_1T_0T_1=T_1T_0T_1T_0$, i.e., Eq.~(\ref{Equ:T_0T_1}) holds.

Finally, thanks to Eq.~(\ref{Equ:theta-T}), for all $j\geq 2$, we have
\begin{eqnarray*}
T_0T_j &=& T_{1}^{-1}\cdots T_{n-1}^{-1}\theta S_0T_j\\
&=&T_{1}^{-1}\cdots T_{n-1}^{-1}T_j\theta S_0\\&=&T_{1}^{-1}\cdots T_{j-2}^{-1}
(T_{j-1}^{-1}T_j^{-1}T_{j-1})T_{j+1}^{-1}\cdots T_{n-1}^{-1}T_j\theta S_0\\
&=&T_{1}^{-1}\cdots T_{n-1}^{-1}T_j\theta S_0\\&=&T_{1}^{-1}\cdots T_{j-2}^{-1}
(T_{j}T_{j-1}^{-1}T_{j}^{-1})T_{j+1}^{-1}\cdots T_{n-1}^{-1}T_j\theta S_0\\
&=&T_jT_0.
\end{eqnarray*}
It completes the proof.
\end{proof}

\begin{remark}
  If $\ell=0$ then the representation $(\Phi, V^{\otimes n})$ reduces to representation defined by Sakamoto and Shoji \cite{SS}. If $m=1$ then the representation $(\Phi, V^{\otimes n})$ reduces to sign $q$-permutation representation of $\mathcal{H}_n(q)$ defined by Moon \cite[Proposition~2.9]{Moon} and Mitsuhashi \cite[Theorem~2.1]{Mit}.
\end{remark}

\begin{remark}
 It is known that Ariki-Koike algebras are cyclotomic quotients of affine Hecke algebras, it would be interesting to define an affine Hecke action on superspaces such that it is stable under the quotient. Moreover, this action would give another way to construct the $\h$-action on $V^{\otimes n}$.
\end{remark}

\begin{remark}\label{Remark:Specialization}Let $q=1$ and $Q_i=\xi^i$, where $\xi$ is a fixed primitive $m$-th root of unity. Then $\h$ reduces to the group algebra $\mathbb{C}W_{m,n}$ of $W_{m,n}$ and the (super) representation $(\Phi, V^{\otimes n})$ of $\h$ reduces to a (super) representation of $\mathbb{C}W_{m,n}$.
\end{remark}

\section{Schur-Sergeev duality}\label{Sec:Schur-Weyl-Cyc}
In this section,  we establish the Schur-Sergeev duality between $U_q(\mathfrak{g})$ and $\h$, which can be viewed as a super analogue of the Schur-Weyl reciprocity for Ariki-Koike algebras given independently by Sakamoto-Shoji \cite{SS} and Hu \cite{Hu}, or as a cyclotomic version of the Schur-Weyl reciprocity between the quantum superalgebra and the Iwahori-Hecke algebra obtained by Moon \cite{Moon} and Mitsuhashi \cite{Mit}.

\begin{lemma}\label{Lemm:q-Comm}For any $X\in \h$ and $Y\in U_q(\mathfrak{g})$,
  $\Phi(X)\Psi^{\otimes n}(Y)=\Psi^{\otimes n}(Y)\Phi(X)$.
\end{lemma}

\begin{proof}Thanks to \cite[Proposition~2.9]{Moon} or  \cite[Theorem~2.1]{Mit}, it is enough to show that the $T_0$-action defined by Eq.~(\ref{Equ:T_0-action}) commutes with the generators of $U_q(\mathfrak{g})$, i.e., commutes with those elements listed in the set $\mathscr{G}$ defined in Eq.~(\ref{Equ:Uq(g)-generators}).
Clearly $S_0$ commutes with $U_q(\fg)$. It reduces to show $S_a$ commutes with elements of $\mathscr{G}$ for all $a\geq 1$. It is easy to check that $S_a$ commutes with $K_j^{\pm1}$ for all $1\leq j\leq d_m$.

 Given $\bi\in\mathcal{I}(k,\ell;n)$, we show that $S_a\Psi^{\otimes n}(E_b)(\bi)=\Psi^{\otimes n}(E_b)S_a(\bi)$ for all $E_b\in \mathscr{G}$. Note that if $c_{a}(\bi)=c_{a+1}(\bi)$ then $c_{a}(\Psi^{\otimes n}(E_b)(\bi))=c_{a+1}(\Psi^{\otimes n}(E_b)(\bi))$. Thus $S_a\Psi^{\otimes n}(E_b)(\bi)=T_a\Psi^{\otimes n}(E_b)(\bi)=\Psi^{\otimes n}(E_b)T_a(\bi)$. If  $c_{a}(\bi)\neq c_{a+1}(\bi)$) then $c_{a}(\Psi^{\otimes n}(E_b)(\bi))\neq c_{a+1}(\Psi^{\otimes n}(E_b)(\bi))$). In this case, we need to show $s_a\Psi^{\otimes n}(E_b)(\bi)=\Psi^{\otimes n}(E_b)s_a(\bi)$. For $p=0,1,\ldots, n-1$, we let \begin{align*}&X_p=\widetilde{K}_b^{\otimes p}\otimes \Psi(E_b)\otimes \mathrm{Id}^{n-1-p}.\end{align*} Then $s_a$ commutes with $X_p$ unless $p=a-1,a$, which implies we only need to show
\begin{eqnarray*}\label{Equ:S-E-comm}
  &&s_a(X_{a-1}+X_{a})(\bi)=(X_{a-1}+X_{a})s_a(\bi).
\end{eqnarray*}
Since $s_a$ affects only the $a$ and $a\!+\!1$-th factors of $\bi$ and the remaining parts are the same for $X_{a}(\bi)$ and $X_{a+1}(\bi)$. We may only consider the $a$ and $a\!+\!1$-th factors, that is, it is enough to verify that
\begin{equation}\label{Equ:Comm-Phi-Psi}
   s_a(\Psi(E_b)\!\otimes\!\mathrm{Id}\!+\!\widetilde{K}_b\!\otimes\!\Psi(E_b))(i_{a}\otimes i_{a\!+\!1})\!=\!(\Psi(E_b)\!\otimes\! \mathrm{Id}\!+\!\widetilde{K}_b\!\otimes\! \Psi(E_b))s_a(i_{a}\!\otimes\! i_{a\!+\!1}).
\end{equation}
Assume that $d_{r-1}<b<d_r$ for some $1\le r\le m$ with $d_0=0$. Thus if
$c_{a}(\bi)$, $r$, $c_{a+1}(\bi)$ are pairwise different, then $\Psi(E_b)=0$, therefore both sides of Eq.~(\ref{Equ:Comm-Phi-Psi}) equal zero; If $c_{a}(\bi)=r\neq c_{a+1}(\bi)$ then $\Psi(E_b)(i_{a+1})=0$ and $\widetilde{K}_a(i_{a+1})=(-1)^{\bar{i}_{a+1}-\bar{i}_a}i_{a+1}$. Then, thanks to Eq.~(\ref{Equ:Tensor-hom}), both sides of Eq.~(\ref{Equ:Comm-Phi-Psi}) equal; The case $c_{a+1}(\bi)=r\neq c_a(\bi)$ can be proved similarly.
So $S_a$ commutes with $\Psi^{\otimes n}(E_b)$ for all $E_b\in \mathscr{G}$. In a similar argument, we can show $S_a$ commutes with $\Psi^{\otimes n}(F_b)$ for all $F_b\in \mathscr{G}$.
  \end{proof}

For positive integers $a$, $b$, $c$, we let
\begin{equation*}
  \Pi(a,b;c)=\left\{(\mu;\nu)\left|\begin{array}{l}\mu\vdash s,\quad\nu\vdash t,\quad s+t=c\\ \ell(\mu)\leq a,\ell(\nu)\leq b, \mu_a\geq \ell(\nu)\end{array}\right.\right\}.
\end{equation*}

\begin{lemma}[\cite{Benkart-Lee} or \protect{\cite[Lemma~5.3]{Moon}}]\label{Lemm:hook=2-partition}
Keeping notations as above, then there is a bijection between $H(a,b;c)$ and $\Pi(a,b;c)$ given by $\gl\mapsto (\gl^1;\gl^2)$, where $\gl^1=(\gl_1, \ldots, \gl_a)$ and $\gl^2=(\gl_1^2,\ldots, \gl_b^2)$ with $\gl_i^2=\max\{\gl_i^*-a,0\}$.
\end{lemma}

Recall that the \textit{Jucys-Murphy elements} of $\h$ are defined recursively by
\begin{equation*}\label{Equ:JM-def}
J_1\!:=g_0\quad \text{ and }\quad J_{i+1}\!:=g_iJ_ig_i, \quad i=1, \cdots, n-1.
  \end{equation*}
It is known that $J_1$, $\ldots, J_n$ generate a maximal commutative subalgebra of $\h$. Let $S^{\bgl}$ be the irreducible $\h$-module corresponding to $\bgl\in \mathscr{P}_{m,n}$. Then $S^{\bgl}$ has a basis $\{v_{\fs}|\fs\in\mathrm{std}(\bgl)\}$ satisfying
\begin{equation}\label{Equ:JM-eigen}
  J_iv_{\fs}=\mathrm{res}_{\fs}(i)v_{\fs}, \quad i=1, \ldots, n,
\end{equation}
for each $\fs\in\mathrm{std(\bgl)}$.
Conversely, if $M$ is an $\h$-module containing a common eigenvector $v_{\fs}$ for $J_1, \ldots, J_n$ satisfying Eq.~(\ref{Equ:JM-eigen}) for some $\fs\in \mathrm{std}(\bgl)$, then the $\h$-submodule $\h v_{\fs}$ of $M$ is a sum of copies of $S^{\bgl}$.
\begin{lemma}\label{Lemm:S-in-V}
  Let $\bgl\in H(\bk|\bl;m,n)$. Then $V^{\otimes n}$ contains an irreducible $\mathcal{H}$-module isomorphic to $S^{\bgl}$ consisting of highest weight vectors of $U_{q}(\fg)$ with highest weight $\bgl$.
\end{lemma}

\begin{proof}
  Thanks to \cite{Moon,Mit}, we have the following $U_q(\fgl(k|\ell))\otimes \mathcal{H}_n(q)$-module  decomposition \begin{eqnarray}\label{Equ:V-decom}
V^{\otimes n}&=&\bigoplus_{\gl\in H(k,\ell;n)}V_{\gl}\otimes S^{\gl},\end{eqnarray}
where $V_{\gl}$ (resp. $S^{\gl}$) is the irreducible $U_{q}(\fg(k|\ell))$-module with highest weight $\gl$ (resp. the irreducible module of $\mathcal{H}_n(q)$ corresponding to $\gl$). Furthermore, the decomposition Eq.~(\ref{Equ:V-decom}) implies for each $\lambda\in H(k,\ell;n)$, the $\mathcal{H}_n(q)$-module $S^{\gl}$ consisting of highest weight vectors for $V_{\gl}$,  that is,  if $\fs$ is a standard $\lambda$-tableau then there is a highest weight vector for $V_{\gl}$.

  Suppose $\bgl=(\gl^{(1)}; \ldots; \gl^{(m)})\!\in\! H(\bk|\bl;m,n)$ and let $n_i=\left|\lambda^{(i)}\right|$.  Thanks to
  Lemma~\ref{Lemm:hook=2-partition}, we may put
  $\gl^{(i)}\!=\!(\mu^{(i)};\nu^{(i)})\!\in\!\Pi(k_i,\ell_i;n_i)$
    for $1\leq i\leq m$.  Now we define a
   standard $\bgl$-tableau $\fs=(\fs^{(1)},\ldots, \fs^{(m)})$ as follows:
   Set $n_{i,1}=\left|\mu^{(i)}\right|$, $n_{i,2}=\left|\nu^{(i)}\right|$,
   $p_{i,1}=n_{i,1}+\cdots+n_{m,1}$, $p_{i,2}=n_{i,2}+\cdots+n_{m,2}$
   for $i=1, \ldots, m$,  and $p_{m+1,1}=p_{m+1,2}=0$. Define $\fs^{(i)}$ by filling
   $n_{i,1}$ numbering $p_{i+1,1}<p_{i+1,1}+2<\cdots<p_{i,1}$ in the boxes in
   $\mu^{(i)}$ first to all boxes of the first row, and then to all the boxes of
   the second row, and so on, in increasing order; then by filling $n_{i,2}$ numbering
   $p_{i+1,2}<p_{i+1,2}+2<\cdots<p_{i,2}$ in the boxes in $\nu^{(i)}$ first to
   all boxes of the first column, and then to all the boxes of the second column, and so on, in increasing order. It is clearly that $\fs^{(i)}$ is a standard $\gl^{(i)}$-tableau.

   Now consider the action of $U_q(\fgl(k_i|\ell_i)\otimes \mathcal{H}_{n_i}(q)$ on the ${V^{(i)}}^{n_i}$ and apply Moon's argument in \cite{Moon}, we can obtain a common eigenvector $w_i\in {V^{(i)}}^{n_i}$ for the Jucys-Murphy elements of $\mathcal{H}_{n_i}(q)$ with respect to $\mathrm{res}_{\fs^{(i)}}$, which is also a highest weight vector of the irreducible  $U_q(\fgl(k_i|\ell_i)$-module $V_{\gl^{(i)}}$ with highest weight $\gl^{(i)}$. Set
\begin{equation*}
  v_{\fs}=w_1\otimes w_{2}\otimes\cdots\otimes w_{m}\in {V^{(1)}}^{\otimes n_1}\otimes {V^{(2)}}^{\otimes n_{2}}\otimes\cdots\otimes {V^{(m)}}^{\otimes n_m}.
\end{equation*}
Then $v_{\fs}$ is a highest weight vector of $V_{\bgl}=V_{\gl^{(1)}}\otimes \cdots\otimes V_{\gl^{(m)}}$. Assume that $a=p_{k+1}<r\leq p_k=b$. We show that $v_{\fs}$ is a common eigenvector for Jucys-Murphy elements of $\h$, that is,
\begin{equation}\label{Equ:JM-eigenvector}
 J_rv_{\fs}=\mathrm{res}_{\fs}(r)v_{\fs}.
\end{equation}
 Clearly $J_rv_{\fs}$ can be written as
\begin{equation*}
 J_rv_{\fs}=T^{-1}_{r}\cdots T^{-1}_{n-1}S_{n-1}\cdots S_1S_0T_1\cdots T_{r-1}v_{\fs}.
\end{equation*}

First note that
\begin{equation}\label{Equ:Ta-Tr-action}
 T_a\cdots T_rv_{\fs}\in {V^{(1)}}^{n_1}\otimes\cdots\otimes {V^{(m)}}^{n_m}.
\end{equation}
Let  $\bi$ be a basis element of $V^{\otimes n}$ occurring in the expression of $ T_a\cdots T_rv_{\fs}$. Then $i_a\in V_r$, and $i_c>i_a$ for all $c<a$. Moreover $i_c\notin V_r$ for $c<a$, which implies
\begin{equation*}
  T_1\cdots T_{a-1}(\bi)=i_a\otimes i_1\otimes\cdots\otimes i_{a-1}\otimes i_{a+1}\otimes \cdots\otimes i_n
\end{equation*}
and $S_{a-1}\cdots S_1S_0T_1\cdots T_{a-1}(\bi)=Q_r\bi$. Thus
\begin{equation*}
J_rv_{\fs}=Q_rT_{r}^{-1}\cdots T_{n-1}S_{n-1}\cdots S_aT_{a}\cdots T_{r-1}v_{\fs}.
\end{equation*}
On the other hand, we have
\begin{equation*}
  S_{b-1}\cdots S_{a}T_a\cdots T_{r-1}v_{\fs}\in V_{1}^{\otimes n_1}\otimes \cdots \otimes V_m^{\otimes n_m}.
\end{equation*}
Let $y=y_1\otimes \cdots\otimes y_n$ be a basis element of $V^{\otimes n}$ occurring in the expression of $S_{b-1}\cdots S_{a}T_a\cdots T_{j-1}v_{\fs}$. Then $y_b\in V_{r}$, $y_i\notin V_{r}$ for all $i>a$. Moreover $y_i<y_{b}$ for all $i>b$. By a similar argument as above, we obtain $T_{b}^{-1}\cdots T_{n-1}S_{n-1}\cdots S_b(y)=y$. Therefore
\begin{eqnarray*}
J_rv_{\fs}&=&Q_rT_{r}^{-1}\cdots T_{n-1}S_{n-1}\cdots S_aT_{a}\cdots T_{r-1}v_{\fs}\\
&=&Q_rT_{r-1}\cdots T_{a}T_{a}\cdots T_{r-1}v_{\fs}\\
&=&\mathrm{res}_{\fs}(r)v_{\fs}.
\end{eqnarray*}
We complete the proof.\end{proof}

Now we compute the multiplicity $m_{\bgl}:=[V^{\otimes n}:V_{\bgl}]$ of $V_{\bgl}$ in $V^{\otimes n}$. Assume that $k_m,\ell_m>1$. We let $\fgl(k_m\!-\!1|\ell_m)$ be the subalgebra of $\fgl(k_m|\ell_m)$ corresponding to the basis $\mathfrak{B}^{(m)}-\{v^{(m)}_{k_m}\}$ and let $\fgl(1,0)$ to be that corresponding to basis $v^{(m)}_{k_m}$. Put $\tilde{\fg}= \fgl(k_1|\ell_1)\oplus\cdots\oplus \fgl(k_m\!-\!1|\ell_m)\oplus \fgl(1|0)$, which is  a subalgebra of $\fg$.

For $\bgl=(\gl^{(1)},\ldots,\gl^{(m)})\in H(\bk|\bl;m,n)$ and $\bmu=(\mu^{(1)},\ldots,\mu^{(m)},\mu^{(m+1)})\in H(\tilde{\bk}|\tilde{\bl};m+1,n)$, where $\tilde{\bk}=(k_1, \ldots,k_m-1,1)$ and $\tilde{\bl}=(\bl,0)$, we write $\bmu\prec\bgl$ if $\gl^{(i)}=\mu^{(i)}$ for $i=1, \ldots, m-1$ and
\begin{equation}\label{Equ:Lamda-mu}
  \gl^{(m)}_1\geq \mu_{1}^{(m)}\geq\cdots\geq \gl^{(m)}_{\ell(\gl^{(m)})-1}\geq \mu_{\ell(\mu^{(m)})-1}^{(m)}\geq \gl^{(m)}_{\ell(\gl^{(m)})},
\end{equation}
where $\gl^{(m)}=(\gl^{(m)}_{1},\ldots, \gl^{(m)}_{\ell(\gl^{(m)})})$ and $\mu^{(m)}=(\mu^{(m)}_{1},\ldots, \mu^{(m)}_{\ell(\gl^{(m)})-1}\geq 0)$.

Notice that $|\bgl|=|\bmu|=n$ and $\ell(\mu^{(m+1)})\leq 1$. Moreover, $\bmu\prec\bgl$ implies that $|\mu^{(m)}|\leq |\gl^{(m)}|$. Thus $\mu^{(m+1)}$ is determined uniquely whenever $\bmu'=(\mu^{(1)},\ldots,\mu^{(m)})$ is determined up to $\prec$.

The following lemma characterize the restriction of $V_{\bgl}$ as a $U_q(\tilde{\fg})$-modules.
\begin{lemma}\label{Lemm:Induction-multiplicity}If $\bgl\in H(\bk|\bl;m,n)$ then $\left.V_{\bgl}\right|_{U_q(\tilde{\fg})}=\bigoplus_{\bmu\prec\bgl}V_{\bmu}$. In particular,
for $\bmu\in H(\tilde{\bk}|\tilde{\bl};m,n)$, we have $$m_{\bmu}=\sum_{\bmu\prec\bgl\in H(\bk|\bl;m,n)}m_{\bgl}.$$
\end{lemma}

\begin{proof}
  Note that the lemma is easily reduced to the case $m=1$, that is, $\fg=\fgl(k|\ell)$ and
$\tilde{\fg}=\fgl(k\!-\!1|\ell)\oplus \fgl(1,0)$. Then $U_q(\fgl(k\!-\!1|\ell))$ is a subalgebra of $U_q(\fgl(k|\ell))$. For $\gl\in H(k,\ell;n)$ and $\mu\in H(k\!-\!1,\ell;n)$ with $\mu\prec \gl$, \cite[Theorem~5.4]{B-Regev} implies
\begin{equation*}
 \mathrm{Res}_{U_q(\fgl(k-1|\ell))}^{U_q(\fgl(k|\ell))} V_{\gl}=\bigoplus_{\mu\prec \gl}V_{\mu}.
\end{equation*}
Note that $U_q(\fg')=U_q(\fgl(k-\!1|\ell))\oplus U_q(\fgl(1|0))$. Thus we yield that
\begin{equation*}
 \mathrm{Res}_{U_q(\tilde{\fg})}^{U_q(\fgl(k|\ell))} V_{\gl} =\bigoplus_{\mu\prec \gl}V_{\mu}\otimes V_{\nu},
\end{equation*}
where $V_{\nu}$ is an irreducible $U_q(\fgl(1,0))$-module with highest weight $\nu$ which is determined uniquely from $\mu$. Since the highest weight $\bmu=(\mu,\nu)$ of $V_{\bmu}=V_{\mu}\otimes V_{\nu}$ satisfies $|\bmu|=n$ and $\ell(\nu)\leq 1$, the condition $\mu\prec\gl$ is equivalent to $\bmu\prec\bgl=\gl$. It completes the proof.
\end{proof}

Denote by $\bar{S}^{\bgl}$ the Specht module of $\mathbb{C}W_{m,n}$ corresponding to $\bgl$. It is well-known that $\dim_{\mathbb{C}}\bar{S}^{\bgl}=\dim_{\mathbb{K}}S^{\bgl}$. Now we can compute the multiplicity $m_{\bgl}$ of $V_{\bgl}$ in $V^{\otimes n}$.

\begin{proposition}\label{Prop:multi}If $\bgl\in H(\bk|\bl;m,n)$ then $m_{\bgl}=\dim_{\mathbb{C}}\bar{S}^{\bgl}$.
  \end{proposition}
\begin{proof} First note that for $\bmu\in H(\tilde{\bk}|\tilde{\bl};m+1,n)$, we have
\begin{equation}\label{Equ:dim-formula}
  \dim S^{\bmu}=\sum_{\bgl\succ \bmu}\dim S^{\bgl}.
\end{equation}
Indeed we can prove this by a similar argument as Lemma~\ref{Lemm:Induction-multiplicity}.

Now we prove the proposition by induction on $\bk|\bl$. First assume that $k_i|\ell_i=1|0$ or $0|1$ for all $i$ and $\bgl=(\gl^{(1)}, \ldots, \gl^{(m)})\in H(\bk|\bl;m,n)$. Then each $\bgl^{(i)}$ can be identified with a non-negative integers, which implies $\dim V_{\bgl}=1$ and $m_{\bgl}=\dim S^{\bgl}$. Assume that there are some $i$ such that $k_i|\ell_i\neq 1|0, 0|1$. Clearly, we may assume that $i=m$. Given $\bgl\in H(\bk|\bl;m,n)$, we choose $\bmu\in H(\tilde{\bk}|\tilde{\bl};m+1,n)$ such that $\bmu\prec \bgl$ with $\mu^{(m)}_i$ for $i=1, \ldots, \ell(\gl^{(m)})-1=d-1$ and $\mu^{(m+1)}=\gl^{(m)}_{d}$. Then for $\widetilde{\bgl}=(\tilde{\gl}^{(1)}, \ldots, \tilde{\gl}^{(m)})\in H(\bk|\bl;m,n)$, $\bmu\prec\widetilde{\bgl}$ implies that $\gl^{(m)}\unlhd\tilde{\gl}^{(m)}$ (see \S\ref{Sec:Dominant-order}), and $\gl^{(i)}=\tilde{\gl}^{(i)}$ for $i=1, \ldots, m-1$.  Now by induction argument, we may assume that $m_{\mu}=\dim S^{\bmu}$ for $\bmu\in H(\tilde{\bk}|\tilde{\bl};m+1,n)$. Therefore
\begin{eqnarray*}
m_{\bmu} &=&\sum_{\widetilde{\bgl}\succ\bmu}\dim S^{\widetilde{\bgl}}=\sum_{\bgl\succ\bmu}m_{\bgl}
=\sum_{\bmu\prec\bgl}\dim S^{\bgl},
\end{eqnarray*}
which implies  $m_{\bgl}=\dim S^{\bgl}$ by applying backward induction on the dominant order $\unlhd$ of weights that $m_{\widetilde{\bgl}}=\dim S^{\widetilde{\bgl}}$ for any $\widetilde{\bgl}\neq \bgl$. We prove the proposition.
\end{proof}

  \begin{theorem}\label{Them:Schur-Weyl}Keeping notations as above, then $\Psi^{\otimes n}(U_q(\mathfrak{g}))$ and $\Phi(\mathcal{H})$ are mutually the fully centralizer algebras of each other, i.e.,
  \begin{eqnarray*}
    \Psi^{\otimes n}(U_q(\mathfrak{g}))=\mathrm{End}_{\h}(V^{\otimes n}), &\qquad& \Phi(\h)=\mathrm{End}_{U_q(\mathfrak{g})}(V^{\otimes n}).
  \end{eqnarray*}
More precisely,   there is a $U_q(\fg)\text{-}\h$-bimodule isomorphism
\begin{equation}\label{Equ:V^n-decom}
  V^{\otimes n}\cong \bigoplus_{\bgl\in\Lambda^+_{\bk|\bl,m}(n)}V(\bgl) \otimes S^{\bgl},
\end{equation}
 where $V(\bgl)$ (resp. $S^{\bgl}$) is the irreducible $U_q(\fg)$ (resp. $\h$)-module indexed by $\bgl$.
\end{theorem}

\begin{proof}Let $\mathscr{A}=\mathrm{End}_{U_q(\fg)}(V^{\otimes n})$ and $\mathscr{B}=\mathrm{End}_{\h}(V^{\otimes n})$. Then $V^{\otimes n}$ is decomposed as a $U_q(\gl)\otimes \mathscr{A}$-module
\begin{equation*}
  V^{\otimes n}=\bigoplus_{\bgl\in H(\bk|\bl;m,n)}V_{\bgl}\otimes \hat{S}_{\bgl},
\end{equation*}
where $\hat{S}_{\bgl}$ is an irreducible $\mathscr{A}$-module corresponding to $\bgl$. According to Lemma~\ref{Lemm:q-Comm}, we have $\Phi(\h)\subseteq \mathscr{A}$. By Lemma~\ref{Lemm:S-in-V} and Proposition~\ref{Prop:multi}, $\hat{S}_{\bgl}$ contains $S^{\bgl}$ as an $\h$-submodule and $\dim \hat{S}_{\bgl}=m_{\bgl}=\dim S_{\bgl}$. Hence $\hat{S}_{\bgl}=S_{\bgl}$ and the decomposition (\ref{Equ:V^n-decom}) follows. It is then clear that $\mathscr{A}=\Phi(\h)$ and $\mathscr{B}=\Psi(U_q(\fg))$.
\end{proof}

We end this paper with some remarks related to the present work.

\begin{remark}\begin{enumerate}

\item Combining Remark~\ref{Remark:Specialization} and Theorem~\ref{Them:Schur-Weyl}, we can obtain a Schur-Sergeev duality between the universal enveloping superalgebra $U(\fg)$  of $\fg$ and the group algebra $\mathbb{C}W_{m,n}$ of $W_{m,n}$, which is a generalization of the Schur-Sergeev duality established in \cite{Serg, B-Regev}.
\item  Based on Shoji's work \cite{S} and Mitsuhashi's work \cite{M2010}, we will give a super Frobenius formula for $\h$ in \cite{Z-F}, which is one of our motivation to construct the Schur-Sergeev duality between quantum superalgebras and cyclotomic Hecke algebras.
 \item In \cite{Z-C}, we will introduce the cyclotomic $q$-Schur superalgebras and show that they enjoy many nice properties. Furthermore, we show  the double centralizer property between the cyclotomic $q$-Schur superalgebra and the Ariki-Koike algebra.

  \item In \cite{Hu}, Hu presented a different proof of the Schur-Weyl reciprocity for the Ariki-Koike algebra. It may interesting to give an alternative proof of the Schur-Sergeev duality for the Ariki-Koike algebra by adapting the \textit{loc. cit.}'s argument.

\item Motivated by Regev's work  \cite{Regev-2013} and the author's work \cite{Zhao}, it may be expect to obtained a Regev formula for Ariki-Koike algebras by applying the Schur-Sergeev for Ariki-Koike algebras.

\item Inspired by Brundan and Kujawa's work \cite{BK} and Du et al's work \cite{DLZ}, it would be very interesting to understand the Mullineux involution for Ariki-Koike algebras \cite{Jacon-L} and the wall-crossing functors for representations of rational Cherednik algebras introduced by Losev in \cite{Losev} via the Schur-Sergeev duality established in the paper.

               \end{enumerate}\end{remark}

\end{CJK*}
\end{document}